\documentclass[12pt,psamsfonts]{amsart}
\usepackage{amsmath}
\usepackage{amsthm}
\usepackage{amssymb}
\usepackage{amscd}
\usepackage{amsfonts}
\usepackage{amsbsy}
\usepackage{graphicx}
\usepackage[dvips]{psfrag}
\usepackage{array}
\usepackage{color}
\usepackage{epsfig}
\usepackage{url}

\newcolumntype{L}{>{\displaystyle}l}
\newcolumntype{C}{>{\displaystyle}c}
\newcolumntype{R}{>{\displaystyle}r}

\renewcommand{\arraystretch}{2}

\newcommand{\R}{\ensuremath{\mathbb{R}}}

\newcommand{\CC}{\mathcal{C}}
\newcommand{\CS}{\ensuremath{\mathcal{S}}}

\newcommand{\CL}{\ensuremath{\mathcal{L}}}

\newcommand{\CO}{\ensuremath{\mathcal{O}}}
\newcommand{\CZ}{\ensuremath{\mathcal{Z}}}
\newcommand{\ov}{\overline}
\newcommand{\la}{\lambda}

\newcommand{\T}{\theta}
\newcommand{\f}{\varphi}
\newcommand{\al}{\alpha}
\newcommand{\be}{\beta}

\newcommand{\s}{\ensuremath{\mathbb{S}}}

\newcommand{\C}{\ensuremath{\mathcal{C}}}

\newcommand{\U}{\ensuremath{\mathcal{U}}}

\newcommand{\x}{\mathbf{x}}

\newcommand{\de}{\delta}

\newcommand{\dis}{\mathrm{dis}}
\newcommand{\dint}{\displaystyle\int}

\def\p{\partial}
\def\e{\varepsilon}

\newtheorem {theorem} {Theorem} 
\newtheorem {proposition} [theorem] {Proposition}

\newtheorem {lemma} [theorem] {Lemma}

\newtheorem {remark} {Remark}
\newtheorem {claim} {Claim}
\newtheorem {mtheorem} {Theorem}

\begin{document}
\renewcommand{\arraystretch}{1.5}

\author[J. Llibre and D.D. Novaes]
{Jaume Llibre$^1$ and Douglas D. Novaes$^{1,2}$}

\address{$^1$ Departament de Matematiques,
Universitat Aut\`{o}noma de Barcelona, 08193 Bellaterra, Barcelona,
Catalonia, Spain} \email{jllibre@mat.uab.cat, ddnovaes@mat.uab.cat}

\address{$^2$ Departamento de Matem\'{a}tica, Universidade Estadual de
Campinas, Rua S\'{e}rgio Buarque de Holanda, 651, Cidade Universit\'{a}ria
Zeferino Vaz, 13083--859, Campinas, SP, Brazil}
\email{dnovaes@ime.unicamp.com}

\title[Periodic solutions in discontinuous differential systems]
{On the periodic solutions of discontinuous\\ piecewise differential systems}

\let\thefootnote\relax\footnotetext{Corresponding author Douglas D. Novaes: Departament de Matematiques,
Universitat Aut\`{o}noma de Barcelona, 08193 Bellaterra, Barcelona,
Catalonia, Spain. Tel. +34 93 5811304, Fax. +34 93 5812790, email: ddnovaes@mat.uab.cat}



\maketitle

\begin{abstract}
Motivated by problems coming from different areas of the applied
science we study the periodic solutions of the following
differential system
\[
x'(t)=F_0(t,x)+\e F_1(t,x)+\e^2 R(t,x,\e),
\]
when $F_0$, $F_1$, and $R$ are discontinuous piecewise functions,
and $\e$ is a small parameter. It is assumed that the manifold $\CZ$
of all periodic solutions of the unperturbed system $x'=F_0(t,x)$
has dimension $n$ or smaller then $n$. The averaging theory is one
of the best tools to attack this problem. This theory is completely
developed when $F_0$, $F_1$ and $R$ are continuous functions, and
also when $F_0=0$ for a class of discontinuous differential systems.
Nevertheless does not exist the averaging theory for studying the
periodic solutions of discontinuous differential system when
$F_0\neq0$. In this paper we develop this theory for a big class of
discontinuous differential systems.
\end{abstract}

\section{Introduction and statement of the main results}

\subsection{Introduction}

The study of the existence of invariant sets, particularly, periodic
solutions is very important for understanding the dynamics of a
differential system. One of the most important tools to detect such
sets is the averaging theory. A classical introduction to this tool
can be found in \cite{V,SVM}.

\smallskip

On the other hand the study of the discontinuous differential systems has it importance and motivation lying in some fields of the applied sciences. Many problems of physics, engineering, economics, and biology are modeled using differential equation with discontinuous right--hand side. For instance we may cite problems in control systems \cite{Bar}, impact and friction mechanics \cite{Br}, nonlinear oscillations \cite{AVK,M}, economics \cite{H,I}, and biology \cite{Ba,Kr}. Recent reviews appeared in \cite{physDspecial, ML}.

\smallskip

Despite to the importance of the discontinuous differential systems mentioned above, there still exist only a few analytical techniques to study the invariant sets of discontinuous differential systems. In \cite{LNT1} the averaging theory has been extended for the following class of discontinuous differential systems
\begin{equation}\label{intro1}
x'(t)=
\begin{cases}
\e F_1(t,x)+\e^2R_1(t,x,\e)\quad \mbox{if}\quad h(t,x)>0,\\
\e F_2(t,x)+\e^2R_2(t,x,\e)\quad \mbox{if}\quad h(t,x)<0.
\end{cases}
\end{equation}
where $F_1,F_2,R_1,R_2$ and $h$ are continuous functions, locally Lipschitz in the variable $x$, $T$--periodic in the variable $t$, and $h$ is a $\CC^1$
function having $0$ as a regular value. The results stated in \cite{LNT1} have been extensively used, see for instance the works \cite{LM1,LM2,N,LLM,LZ}.

\smallskip

In this paper we focus on the development and improvement of the averaging theory for studying periodic solutions of a much bigger class of discontinuous differential systems than in \eqref{intro1}. Regarding to the averaging theory for finding periodic solutions there are essentially three main theorems. In what follows we describe these theorems.

\smallskip

The first one is concerning about the study of the
periodic solutions of the periodic differential systems of the form
\[
x'=\e F_1(t,x)+\e^2 F_2(t,x)+\cdots+\e^m F_m(t,x)+\e^{m+1} R(t,x,\e),
\]
with $x\in\R^d$. For continuous differential systems, even for the non--differentiable ones, this theory is already completely developed (see for instance \cite{V,SVM,BL,GGL,LNT2}), and for discontinuous differential systems this theory is develop up to order $2$ in $\e$ (see \cite{LNT1,LMN}).

\smallskip

The other two theorems go back to the works of Malkin \cite{Ma} and Roseau \cite{Ro}. They studied the periodic solutions of the periodic differential systems of the form
\[
x'=F_0(t,x)+\e F_1(t,x)+\e^2 F_2(t,x)+\cdots+\e^m F_m(t,x)+\e^{m+1} R(t,x,\e),
\]
with $x\in\R^d$, distinguishing when the manifold $\CZ$ of all periodic solutions of the unperturbed system $x'=F_0(t,x)$ has dimension $d$ or smaller
then $d$. These theories are well developed for continuous differential systems (see for instance \cite{BFL,BGL,GGL,LNT2}). Nevertheless there is no theory for studying such problems in discontinuous differential systems. Thus our main objective in this paper is to develop these last theorems for a big class of discontinuous differential systems.

\smallskip

In subsections \ref{Prel} and \ref{stat} we describe the class of discontinuous differential systems that we shall consider in this paper together with our main results, and we also do an application. In section \ref{s2} we prove our main results, and in section \ref{PP} we describe carefully the application of our results.

\subsection{Preliminaries}\label{Prel}

We take the ODE's
\begin{equation}\label{A1}
x'(t)=F^n(t,x), \quad (t,x)\in \s^1 \times D \quad \textrm{for} \quad
n=1,2,\ldots,N,
\end{equation}
where $D\subset\R^d$ is an open subset and $\s^1=\R/T$ for some positive real number $T$.
Here $F^n: \s^1\times D \to \R^d$ for $n=1,2,\ldots,N$ are continuous functions and the prime denotes derivative with respect to the time $t$. For $n=1,2,\ldots,N$ let $S_n$ be open connected and disjoint subsets of $\s^1\times D$. The boundary of $S_n$ for $n=1,2,\ldots,N$ is assumed to be piecewise
$\CC^m$ embedded hypersurface with $m\geq1$ and the union of all these boundaries is denoted by $\Sigma$. Moreover we assume that $\Sigma$ and all $S_n$
together cover $\s^1\times D$. We call an {\it $N$--Discontinuous
Piecewise Differential System}, or simply a {\it DPDS}, when the context is clear, the following differential system
\begin{equation}\label{s1}
x'(t)=\left\{\begin{array}{LC}
F^1(t,x) & (t,x)\in \ov{S_1},\\
F^2(t,x) &(t,x)\in \ov{S_2},\\
&\vdots\\
F^N(t,x) & (t,x)\in \ov{S_N}.
\end{array}\right.
\end{equation}
Here $\ov{S_n}$ denotes the closure of $S_n$ in $\s^1\times D$.

\smallskip

Instead of working with system \eqref{s1} we can work with the following associated system.
\begin{equation}\label{ss1}
x'(t)=F(t,x)=\sum_{n=1}^N \chi_{\ov {S_n}}(t,x) F^n(t,x), \quad
(t,x)\in \s^1\times D,
\end{equation}
where for a given subset $A$ of $\s^1\times D$ the {\it
characteristic function} $\chi_A(t,x)$ is defined as
\[
\chi_A(t,x)=
\begin{cases}
1 \quad\textrm{if}\quad (t,x)\in A,\\
0 \quad\textrm{if}\quad (t,x)\notin A.
\end{cases}
\]
Systems \eqref{s1} and \eqref{ss1} does not coincides in $\Sigma$. Indeed system \eqref{s1} is multivalued in $\Sigma$ whereas system \eqref{ss1} is single valued in $\Sigma$. Using Filippov's convention for the solutions of the systems \eqref{s1} or \eqref{ss1} (see \cite{Fi}) passing through a point $(t,x)\in\Sigma$ we have that these solutions do not depend on the value $F(t,x)$. So the solutions of systems \eqref{s1} and \eqref{ss1} are the same.

\smallskip

When $F^n$ for $n=1,2,\ldots,N$ are $\CC^1$ functions we define the ``derivative'' of the discontinuous piecewise differentiable function $F(t,x)$ with respect to $x$ as
\begin{equation}\label{nota}
D_x F(t,x)=\sum_{n=1}^N \chi_{\ov {S_n}}(t,x) D_x F^n(t,x).
\end{equation}
We note that when the function $F(t,x)$ is differentiable with respect to the variable $x$ then the above definition coincides with the usual derivative.



\smallskip

We say that a point $p\in\Sigma$ is a {\it generic point of
discontinuity} if there exists a neighborhood $G_p$ of $p$ such that
$\CS_p=G_p\cap\Sigma$ is a $\CC^m$ embedded  hypersurface in $\s^1\times D$ with $m\geq 1$, such that the hypersurface $\CS_p$ splits
$G_p\backslash\CS_p$ in two disconnected regions, namely $G_p^+$ and
$G_p^-$, and the vector fields $F_p^+=F|_{G_p^+}$ and $F_p^-=F|_{G_p^-}$ are continuous. We define $l(p)$ as the segment connecting the
vectors $F_p^+(p)$ and $F_p^-(p)$ when these have the same origin (see Figures \ref{fig1} and
\ref{fig2}). 

\smallskip

Let $\CS\subset \Sigma$ be an embedded hypersurface in $\s^1\times D$ and $T_p \CS$ denotes the tangent space of $\CS$ at the point $p$. In what follows we define the {\it crossing region} $\Sigma^c(\CS)$ (see Figure \ref{fig1}), and the {\it sliding region} $\Sigma^s(\CS)$ (see Figure \ref{fig2}) of the hypersurface$\CS$.
\[
\begin{array}{CCC}
\Sigma^c(\CS)=\left\{p\in \CS:\, l(p)\cap T_p \CS=\emptyset \right\}
&\textrm{and}& \Sigma^s(\CS)=\left\{p\in \CS:\, l(p)\cap T_p
\CS\neq\emptyset \right\}.
\end{array}
\]
These definitions only make sense when the linear space $T_p\CS$ is based at the origin of the vectors $F^+_p(p)$ and $F^-_p(p)$.

\smallskip

The hypersurface $\CS\subset\Sigma$ can be decomposed as the
union of the closure of its {\it crossing region} with its {\it sliding region}. 

\smallskip

When the hypersurface $\CS\subset\Sigma$ is given by $\CS=h^{-1}(0)$ for some $\CC^1$ function $h:\s^1\times D\rightarrow\R$ having $0$ as a regular value, then the above definitions becomes
\[
\begin{array}{L}
\Sigma^c(\CS)=\left\{p\in \CS: \langle \nabla
h(p),(1,F^+(p))\rangle\langle \nabla h(p),(1,F^-(p))\rangle>0 \right\}\quad
\textrm{and}\\ \Sigma^s(\CS)=\left\{p\in \CS:\, \langle \nabla
h(p),(1,F^+(p))\rangle\langle \nabla h(p),(1,F^-(p))\rangle<0  \right\}.
\end{array}
\]

\begin{figure}[h]
\centering
\psfrag{T}{$T_p\Sigma$}
\psfrag{M}{$\CS_p$}
\psfrag{p}{$p$}
\psfrag{vj}{$F_p^-(p)$}
\psfrag{vi}{$F_p^+(p)$}
\psfrag{l}{$l(p)$}
\psfrag{Aj}{$G_p^-$}
\psfrag{Ai}{$G_p^+$}
\includegraphics[width=14cm]{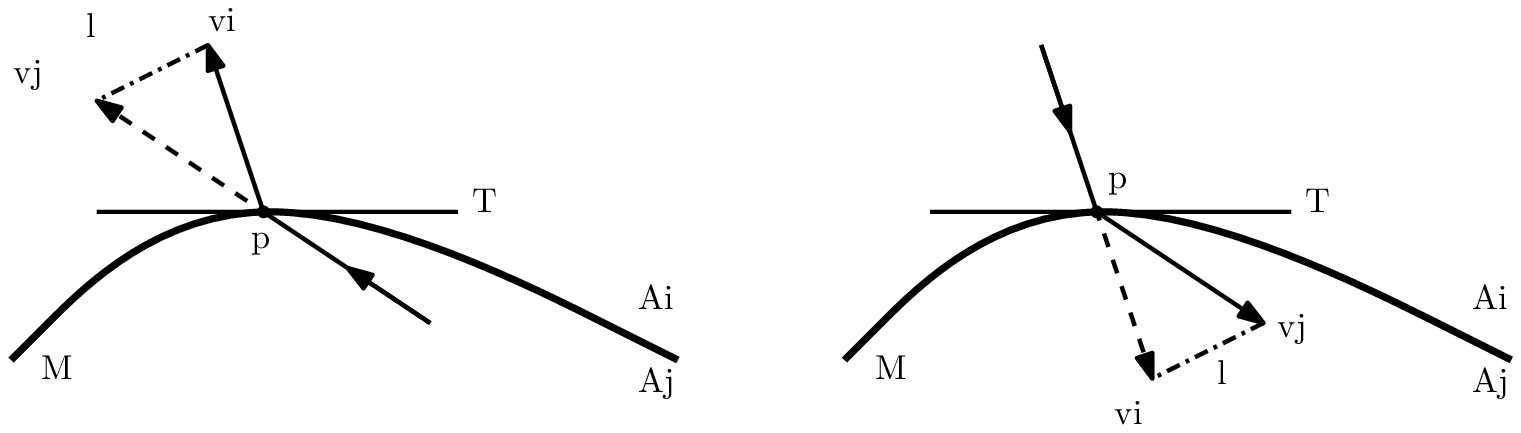}
\vskip 0cm \centerline{} \caption{\small \label{fig1} Crossing
region of $\CS$: $\Sigma^c\CS$. }
\end{figure}

\begin{figure}[h]
\centering
\psfrag{T}{$T_p\Sigma$}
\psfrag{M}{$\CS_p$}
\psfrag{p}{$p$}
\psfrag{vj}{$F_p^-(p)$}
\psfrag{vi}{$F_p^+(p)$}
\psfrag{l}{$l(p)$}
\psfrag{Aj}{$G_p^-$}
\psfrag{Ai}{$G_p^+$}
\includegraphics[width=14cm]{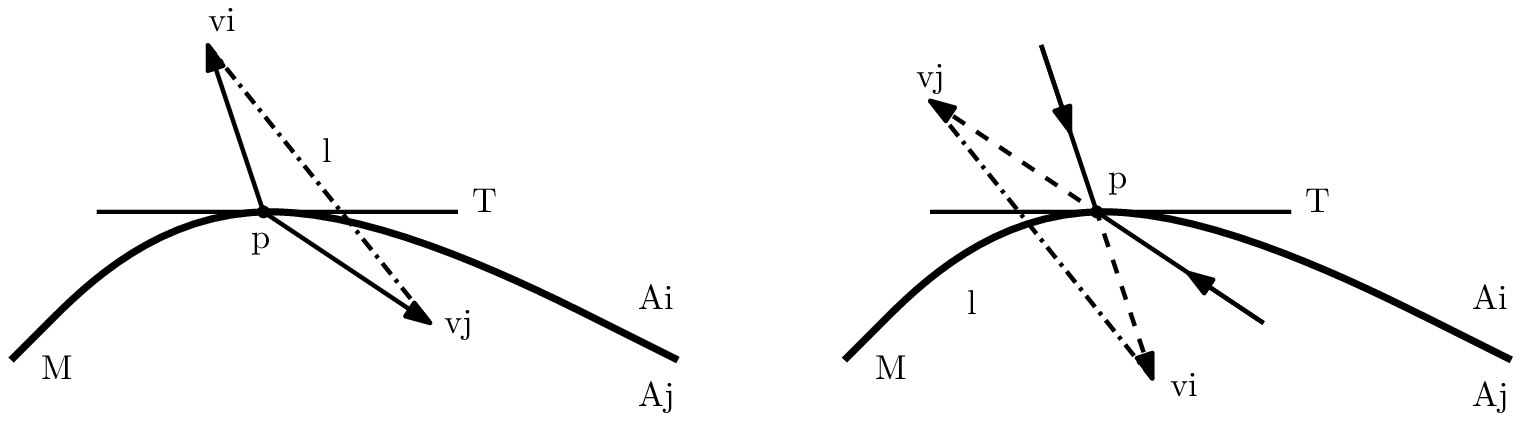}
\vskip 0cm \centerline{} \caption{\small \label{fig2} Sliding region
of $\CS$: $\Sigma^s\CS$. }
\end{figure}

\smallskip

Globally we define the {\it crossing region} $\Sigma^c$ as the generic
points of discontinuity $p$ such that $p\in\Sigma^c(\CS_p)$. The {\it sliding region} $\Sigma^s$ is defined analogously. Later on this paper for a point $p\in \Sigma^c$ we shall denote $T_p\Sigma=T_p\CS_p$.

\smallskip

Let $\f_{F^n}(t,q)$ be the solution of system \eqref{A1} passing
through the point $q\in S_n$ at time $t=0$, i.e. $\f_{F^n}(0,q)=
q$. The local solution $\f_{F}(t,q)$ of system \eqref{ss1} passing
through a point $p\in\Sigma^c$ at time $t=0$ is given by the
Filippov convention, i.e. for $p\in\Sigma^c$ such that $l(p)\subset
G_p^+$ and taking the origin of time at $p$, the trajectory through
$p$ is defined as $\f_F(t,p)=\f_{F_p^-}(t,p)$ for $t\in
I_p\cap\{t<0\}$, and $\f_F(t,p)=\f_{F_p^+}(t,p)$ for $t\in
I_p\cap\{t>0\}$. Here $I_p$ is an open interval having
the $0$ in its interior. For the case $l(p)\subset G_p^-$ the definition is
the same reversing the time.

\smallskip

Assuming that the functions $F^n(t,x)$ are Lipschitz in the variable $x$ for $n=1,2,\ldots,N$, the results on Filippov systems (see \cite{Fi}) guarantee the uniqueness of the solutions reaching the set of discontinuity only at points of $\Sigma^c$.

\subsection{Statements of the main results}\label{stat}

Let $D$ be an open subset of $\R^d$ and for $n=1,2,\ldots,N$ let $F_0^n:\s^1\times D\rightarrow\R^d$ be a $\C^m$ function with $m\geq1$, and $F_1^n:\s^1\times D\rightarrow\R^d$, and $R^n:\s^1\times D\times[0,1]\rightarrow\R^d$ be continuous functions which are Lipschitz in the second variable. All these functions can be seen as $T$--periodic functions in the variable $t$ when $t\in\R$. Later on in this paper we shall assume more conditions under these functions.

\smallskip

Now taking
\[
\begin{aligned}
&F_i(t,x)=\sum_{n=1}^N \chi_{\ov {S_n}}(t,x) F_i^n(t,x), \quad
\textrm{for $i=0,1$, and}\\ &R(t,x,\e)=\sum_{n=1}^N \chi_{\ov
{S_n}}(t,x) R^n(t,x),
\end{aligned}
\]
we consider the following DPDS,
\begin{equation}\label{MRs1}
x'(t)=F_0(t,x)+\e F_1(t,x)+\e^2R(t,x,\e).
\end{equation}
The parameter $\e$ is assumed to be small. We recall that $\Sigma$ denotes the union of the boundaries of $S_n$ for $n=1,2,\ldots,N$. 

\smallskip

In order to present our main results we have to introduce more definitions and notation.

\smallskip

For $z\in D$  and $\e>0$ sufficiently small we denote by $x(\cdot,z,\e):[0,t_{(z,\e)})\rightarrow \R^d$ the solution of system \eqref{MRs1} such that $x(0,z,\e)=z$. Given a subset $B$ of $D$ we define $\widetilde{B}^{\e}=\ov{\{(t,x(t,z,\e)):\,z\in B, t\in [0,t_{(z,\e)})\}}$. 

\smallskip

We denote by $\Sigma_0$ the set of points $x\in D$ such that the function $F(0,x)$ is discontinuous, clearly $\{0\}\times\Sigma_0\subset\Sigma$. 

\smallskip

One of the main hypothesis of this paper is that the unperturbed system
\begin{equation}\label{ups}
x'(t)=F_0(t,x),
\end{equation}
has a manifold $\CZ$ embedded in $D\backslash \p \Sigma_0$ such that the solutions starting in $\CZ$ are all $T$--periodic functions and reach the set of discontinuity $\Sigma$ only at its crossing region $\Sigma^c$. Here $\p \Sigma_0$ denotes the boundary of $\Sigma_0$ with respect to topology of $D$. Precisely,
\begin{itemize}
\item[($H$)] let $\CZ=\{z_{\al}=(\al,\be_0(\al)):\,\al\in\ov V\}$, where $V$ is an open and bounded subset of $\R^k$, and $\beta_0:\ov V\rightarrow\R^{d-k}$ is a $\C^m$ function with $m\geq 1$. We shall assume that $\CZ\subset D$, $\CZ\cap \p\Sigma_0=\emptyset$, $\widetilde{\CZ}^0\cap\Sigma\subset\Sigma^c$ and for each $z_{\al}\in \CZ$ the unique solution $x_{\al}(t)=x(t,z_{\al},0)$ is $T$--periodic. 
\end{itemize}

\begin{remark}
Suppose that the solution $x_{\al}(t)$ reaches the set $\Sigma^c$ $\kappa_{\al}$ times. The assumption $\CZ\cap \p\Sigma_0=\emptyset$ in hypothesis $(H)$ implies that for each $z_{\al}\in \CZ$ there exists a small neighborhood $U_{\al}\subset D$ of $z_{\al}$ such that for $\e>0$ sufficiently small every solution of the perturbed system \eqref{MRs1} starting in $U_{\al}$ reach the crossing region of the set of discontinuity $\Sigma^c$ also $\kappa_{\al}$ times. This fact will be well justified in the proofs of Lemmas \ref{l2} and \ref{l3} in section \ref{s2}
\end{remark}


 
For $z\in D$ we take the following discontinuous piecewise linear differential system
\begin{equation}\label{lin}
y'=D_xF_0(t,x(t,z,0))\,y,
\end{equation}
which can be seen as the linearization of the unperturbed system \eqref{ups} along the solution $x(t,z,0)$. We note that for each $z\in D$ the function $t\mapsto D_xF_0(t,x(t,z,0))$ is piecewise $\C^m$ with $m\geq 1$, so we can consider a fundamental matrix $Y(t,z)$ of the differential system \eqref{lin}. 
Clearly $t\mapsto Y(t,z)$ is continuous piecewise $\C^m$ function. We define
\begin{equation}\label{y1}
y_1(t,z)=Y(t,z)\int_0^t Y(s,z)^{-1}F_1(s,x(s,z,0))ds.
\end{equation}
Now for $z_{\al}\in \CZ$ we denote $Y_{\al}(t)=Y(t,z_{\al})$. Let $\pi:\R^k\times \R^{d-k}\rightarrow \R^k$ and $\pi^{\perp}:\R^k\times \R^{d-k}\rightarrow \R^{d-k}$ be the projections onto the first $k$ coordinates and onto the last $d-k$ coordinates, respectively. Thus we define the averaged function $f_1: \ov V\rightarrow\R^k$ as
\begin{equation}\label{f1}
f_1(\al)=\pi y_1(T,z_{\al}).
\end{equation}

In what follows $\dis(x,A)$ denotes the Hausdorff distance function between a point $x\in D$ and a set $A\subset D$, and as usual the function $d_B(f_1,W,0)$ denotes the Brouwer degree (see for instance \cite{B} for details on the Brouwer degree). Our main result on the periodic solutions of DPDS \eqref{MRs1} is the following.

\begin{mtheorem}\label{MRt1}
In addition to the hypothesis $(H)$ we assume that
\begin{itemize}
\item[$(H1)$] for $n=1,2,\ldots,N$, the functions
$F_0^n$ and $\beta_0$ are of class $\CC^1$; the continuous functions $D_x F_0^n$, $F_1^n$ and $R$ are locally
Lipschitz with respect to $x$; and the boundary of $S_n$ are piecewise $\CC^1$ embedded hypersurface in $\R\times D$;

\item[$(H2)$] for any $\al\in\ov V$ there exists a fundamental matrix solution $Y(t,z)$ of \eqref{lin} such that the matrix $Y_{\al}(T)Y_{\al}(0)^{-1}-Id$ has in the upper right corner the null $k\times(d-k)$ matrix, and in the lower right corner has the $(n-k)\times(n-k)$ matrix $\Delta_{\al}$ with $\det(\Delta_{\al})\neq0$;

\item[$(H3)$]  for an open subset $U$ of $D$ such that $\CZ\subset U$ we have that $(0,y_1(s,z))\in T_{(s,x(s,z,0))}\Sigma$ whenever $(s,x(s,z,0))\in\Sigma^c$ for $(s,z)\in\s^1\times U$;

\item[$(H4)$] there exists $W$ open subset of $V$ such that $f_1(\al)\neq 0$ for $\al\in\p W$ and $d_B(f_1,W,0)\neq 0$.
\end{itemize}
Then for $\e>0$ sufficiently small, there exists a $T$--periodic
solution $\f(t,\e)$ of system \eqref{MRs1} such that $\dis(\f(0,\e),\CZ)\to
0$ as $\e\to 0$.
\end{mtheorem}

Theorem \ref{MRt1} is proved in Section \ref{s2}.

\begin{remark}\label{bc}
When $f_1$ is a $\C^1$ function the assumption
\begin{itemize}
\item[(h4)] there exists $a\in V$ such that $f_1(a)=0$ and $\det(f_1'(a))\neq0$,
\end{itemize}
is a sufficient condition to guarantees the validity of the hypothesis $(H4)$.
\end{remark}

\begin{mtheorem}\label{MRt2}
We suppose that the hypotheses $(H)$, $(H2)$ and $(H3)$ of Theorem \ref{MRt1} hold. If we assume that 
\begin{itemize}
\item[$(h1)$] for $n=1,2,\ldots,N$, $F_0^n$, $D_x F_0^n$, $F_1^n$, $R^n$, and $\beta_0$ are $\CC^2$ functions and the boundary of $S_n$ are piecewise $\CC^2$ embedded hypersurface in $\R\times D$,
\end{itemize}
then $f_1(\al)$ is a $\CC^1$ function  for every $\al\in\ov V$. Moreover, if we assume in addition that hypothesis $(h4)$ holds, then for $\e>0$ sufficiently small, there exists a $T$--periodic
solution $\f(t,\e)$ of system \eqref{MRs1} such that $\f(0,\e)\to
z_a$ as $\e\to 0$.
\end{mtheorem}


\smallskip

In what follows we provide an application of Theorems \ref{MRt1} and \ref{MRt2}. We study the existence of limit cycles which bifurcate from the periodic solutions of the linear differential system $(\dot{u},\dot{v},\dot{w})=(-v,u,w)$ perturbed inside the class of all discontinuous piecewise linear differential systems with two zones separated by the plane $\Sigma=\{v=0\}\subset\R^3$, i.e.
\begin{equation}\label{lep}
\left(\begin{array}{C}
\dot{u}\\
\dot{v}\\
\dot{w}\\
\end{array}\right)=\left\{\begin{array}{L}

\left(\begin{array}{L}
-v+\e(a_1^++b_1^+ u+c_1^+ v+d_1^+ w)\\
u+\e(a_2^++b_2^+ u+c_2^+ v+d_2^+ w)\\
w+\e(a_3^++b_3^+ u+c_3^+ v+d_3^+ w)
\end{array}\right) \quad \textrm{if} \quad v>0,\vspace{0.2cm}\\

\left(\begin{array}{L}
-v+\e(a_1^-+b_1^- u+c_1^- v+d_1^- w)\\
u+\e(a_2^-+b_2^- u+c_2^- v+d_2^- w)\\
w+\e(a_3^-+b_3^- u+c_3^- v+d_3^- w)
\end{array}\right) \quad \textrm{if} \quad v<0.

\end{array}\right.
\end{equation}

Our result on the existence of a limit cycle of system \eqref{lep} is the following.

\begin{proposition}\label{p1}
If $(a_2^--a_2^+)(b_1^- +b_1^+ +c_2^- +c_2^-)>0$, then for $|\e|>0$ sufficiently small there exists a periodic solution $(u(t,\e),v(t,\e),w(t,\e))$ of system \eqref{lep} such that $w(0,\e)\to 0$  when $\e\to 0$. Moreover, we can find $(u^*,v^*)\in\R^2$ such that 
\[
||(u^*,v^*)||=\dfrac{4(a_2^- - a_2^+)}{\pi(b_1^-+b_1^++c_2^-+c_2^+)},
\] 
and $(u(0,\e),v(0,\e))\to(u^*,v^*)$ when $\e\to 0$.
\end{proposition}

Proposition \ref{p1} is proved in Section \ref{PP}.

\section{Proof of Theorem \ref{MRt1}}\label{s2}

Before proving our main result we state some
preliminary lemmas. 


Given a function $\xi:[0,1]\rightarrow\R^d$ we say that $\xi(\e)=\CO(\e^{\ell})$ for some positive integer $\ell$
if there exists constants $\e_1>0$ and $k>0$ such that $||\xi(\e)
||\leq k|\e^{\ell}|$ for $0\leq\e\leq\e_1$, and that $\xi(\e)=o(\e^{\ell})$ for some positive integer $\ell$ if
\[
\lim_{\e\to 0}\dfrac{||\xi(\e)||}{\e^{\ell}}=0.
\]
Here $||\cdot||$ denotes the usual Euclidean norm of $\R^d$. The symbols $\CO$ and $o$ are called the {\it Landau's symbols} (see
for instance \cite{SVM}).

\begin{lemma}\label{l1}
Under the hypotheses $(H)$, $(H1)$, and $(H3)$ of Theorem \ref{MRt1} 
there exist an open and bounded subset $C$ of $U\backslash\p\Sigma_0$, a compact subset $Z\subset C$ with $\CZ\subset Z^{\circ}$, and a small parameter $\e_0>0$ such that $t_{(z,\e)}>T$ for every $z\in C$ and $\e\in[0,\e_0]$. Moreover $x(t,z,\e)=x(t,z,0)+\e y_1(t,z)+o(\e)$ for every $z \in Z$, $\e\in[0,\e_0]$, and $t\in[0,T]$. Here $Z^{\circ}$ denotes the interior of the set $Z$ with respect to the topology of $D$, and the function $y_1$ is given in \eqref{y1}.
\end{lemma}
\begin{proof}

We note that $\CZ$ is a compact subset of $D$ and $\p\Sigma_0$ is a closed subset of $D$, such that, from the hypothesis $(H)$, $\CZ\cap\p\Sigma_0=\emptyset$. So there exists an open subset $A$ of $D$ such that $\CZ\subset A$ and $\ov A\cap \p\Sigma_0=\emptyset$.

\smallskip

Also from hypothesis $(H)$ we have that for $\al\in \ov V$ the continuous function $x_{\al}(t)$ reaches the set $\Sigma$ only at points of $\Sigma^c$. Since this function is $T$--periodic we can find a finite sequence $(t^i_{\al})$ for $i=0,1,\ldots,\kappa_{\al}$ with $t_{\al}^0=0$ and $t_{\al}^{\kappa_{\al}}=T$ such that 
\[
x_{\al}(t)=
\begin{cases}
\begin{array}{CCRL}
x^1_{\al}(t) &\textrm{if}&0=&t^0_{\al}\leq t\leq t^1_{\al},\\
x^2_{\al}(t)  &\textrm{if}&&t^1_{\al}\leq t\leq t_{\al}^2,\\
\vdots\\
x^i_{\al}(t)  &\textrm{if}&&t^{i-1}_{\al}\leq t\leq t^i_{\al},\\
\vdots\\
x^{\kappa_{\al}}_{\al}(t) &\textrm{if}&&t^{\kappa_{\al}-1}_{\al}\leq t\leq t_{\al}^{\kappa_{\al}}=T,\\
\end{array}
\end{cases}
\]
where each curve $t\mapsto x_{\al}^i(t)=x^i(t,z_{\al},0)$ reaches the set $\Sigma^c$ only at $t=t^{i-1}_{\al}$ and $t=t^{i}_{\al}$ for $i=2,3,\ldots,\kappa_{\al}-1$, the curve $x_{\al}^1$ reaches the set $\Sigma^c$ only at $t=0$ and $t=t^{1}_{\al}$ if $(0,z_{\al})\in\Sigma$, and only at $t=t^{1}_{\al}$ if $(0,z_{\al})\notin\Sigma$, the curve $x_{\al}^{\kappa_{\al}}$ reaches the set $\Sigma^c$ only at $t=t^{\kappa_{\al}-1}_{\al}$ and $t=T$ if $(T,x(T,z_{\al},0))\in\Sigma$, and only at $t=t^{\kappa_{\al}-1}_{\al}$ if $(T,x(T,z_{\al},0))\notin\Sigma$. From the definition of the crossing region $\Sigma^c$ these intersections are transversely.

\smallskip

Since $x_{\al}^i$ for $i=1,2,\ldots,\kappa_{\al}$ are solutions of Lipschitz differential equations we can use the results of continuous dependence of the solutions on initial conditions and parameters to ensure the existence of a small parameter $\e_{\al}$ and a small neighborhood $C_{\al}\subset A\cap U$ of $z_{\al}$ such that $\widetilde{C_{\al}}^{\e}\cap\Sigma\subset\Sigma^c$ for every $\e\in[0,\e_{\al}]$. From the compactness of $\CZ$ we can choose $\e_1$ as a minimum element of $\e_{\al}\in\ov V$. Now taking $C=\cup_{\al\in\ov V} C^{\al}$ it follows that $\widetilde{C}^{\e}\cap\Sigma\subset\Sigma^c$ for every $\e\in[0,\e_1]$. Moreover, we can take $\e_1>0$ and $C$ smaller in order that the function $t\mapsto x(t,z,\e)$ is defined for all $(t,z,\e)\in \s^1\times \ov C\times[0,\e_1]$. This is again a simple consequence of the Theorem of continuous dependence on initial conditions and parameters. 

\smallskip


\smallskip

Thus for $z\in \ov C$ and $\e\in[0,\e_1]$ the function
$t\mapsto x(t,z,\e)$ is continuous and piecewise
$\C^1$. So we can find a finite sequence $(t^i(z,\e))$ for $i=0,1,\ldots\kappa_z$ with $t^1(z,\e)=0$ and $t^{\kappa_{z}}(z,\e)=T$ such that
\begin{equation}\label{xx}
x(t,z,\e)=
\begin{cases}
\begin{array}{CCRL}
x^1(t,z,\e) &\textrm{if}&0=&t^0(z,\e)\leq t\leq t^1(z,\e),\\
x^2(t,z,\e) &\textrm{if}&&t^1(z,\e)\leq t\leq t^2(z,\e),\\
\vdots\\
x^i(t,z,\e) &\textrm{if}&&t^{i-1}(z,\e)\leq t\leq t^i(z,\e),\\
\vdots\\
x^{\kappa_z}(t,z,\e) &\textrm{if}&&t^{\kappa_z-1}(z,\e)\leq t\leq t^{\kappa_z}(z,\e)=T,\\
\end{array}
\end{cases}
\end{equation} 
for which we have the following recurrence
\begin{equation}\label{rec}
\begin{array}{CCC}
x^1(0,z,\e)=z &\textrm{and}& x^{i}(t^{i-1}(z,\e),z,\e)=x^{i-1}(t^{i-1}(z,\e),z,\e),
\end{array}
\end{equation}
for $i=2,3,\ldots,\kappa_z$. The crossing region $\Sigma^c$ is an open subset of $\Sigma$, so for each $z\in \ov C$ we can find $0<\e_z\leq \e_1$ sufficiently small such that the number $\kappa_z$ of intersections between the curve $t\mapsto x(t,z,\e)$ with the set $\Sigma^c$ for $0\leq t\leq T$ and for $\e\in[0,\e_z]$ does not depend of $\e$. Since $\ov C$ is compact we can find $\e_2$ a minimum element of the $\e_z$'s for $z\in\ov C$ such that the above statement holds for every $z\in\ov C$ and $\e\in[0,\e_2]$.

\smallskip

Here again for every $z\in\ov C$ and $\e\in[0,\e_2]$ each curve $t\mapsto x^i(t,z,\e)$ reaches the set $\Sigma^c$ only at $t=t^{i-1}(z,\e)$ and $t=t^{i}(z,\e)$ for $i=2,3,\ldots,\kappa_{z}-1$, the curve $x^1(t,z,\e)$ reaches the set $\Sigma^c$ only at $t=0$ and $t=t^{1}(z,\e)$ if $(0,z)\in\Sigma$, and only at $t=t^{1}(z,\e)$ if $(0,z)\notin\Sigma$, the curve $x^{\kappa_{z}}(t,z,\e)$ reaches the set $\Sigma^c$ only at $t=t^{\kappa_{z}-1}(z,\e)$ and $t=T$ if $(T,x(T,z,0))\in\Sigma$, and only at $t=t^{\kappa_{z}-1}(z,\e)$ if $(T,x(T,z,0))\notin\Sigma$

\smallskip

The functions $t\mapsto x^i(t,z,\e)$ for $i=1,2,\ldots,\kappa_z$ are $\C^1$ and satisfy the DPDS \eqref{MRs1}, so there exists a
subsequence $(n_i)$ for $i=1,\ldots,\kappa_z$ with $n_i\in\{1,2,\ldots,N\}$ such that
\begin{equation}\label{xi}
\dfrac{\p}{\p t}x^i(t,z,\e)=F_0^{n_i}(t,x^i(t,z,\e))+\e F_1^{n_i}(t,x^i(t,z,\e))+\e^2R^{n_i}(t,x^i(t,z,\e),\e).
\end{equation}
Therefore the function $x^i(t,z,\e)$ is the solution of the {\it Cauchy
Problem} defined by the differential system $\eqref{xi}$ together
with the corresponding initial condition given in \eqref{rec}. Moreover $x^i(t,z_{\al},0)=x^i_{\al}(t)$ and $t^i(z_{\al},0)=t_{\al}^i$ for $i=1,2,\ldots,\kappa_z$.

\smallskip

From the continuity of the function $x(t,z,\e)$ we can choose a compact subset $K$ of $D$ such that $x(t,z,\e)\in K$ for all
$(t,z,\e)\in\s^1\times \ov{C}\times[0,\e_2]$. From the continuity of the functions $F_i^n$ and $R^n$ for $i=0,1$ and $n=1,2,\ldots,N$ we have that these functions are bounded on the compact set $\s^1\times K\times[0,\e_2]$, so let $M$ be an upper bound for all these functions. Let $L$ be being the maximum Lipschitz constant of the functions $F_i^n$, $DF_0^n$, and $R^n$ for $i=0,1$ and $n=1,2,\ldots,N$ on the compact set $\s^1\times K\times[0,\e_2]$.

\smallskip

We compute
\[
\left|\left|\int_0^t
R(s,x(s,z,\e),\e)ds\right|\right|\leq\int_0^T||R(s,x(s,z,\e),\e)
||ds=TM,
\]
which implies that $\dint_0^t R(s,x(s,z,\e),\e)ds=\CO(1)$ in the parameter $\e$.

\smallskip

For $z\in\ov C$ and $t\in(0,T)$ we can find
$\ov{\kappa}\in\{1,2,\ldots,\kappa_z-1\}$ such that
$t\in[t^{\ov{\kappa}-1}(z,\e),t^{\ov{\kappa}}(z,\e))$ and
\[
\begin{array}{RL}
x(t,z,\e)=&x^{\ov{\kappa}}(t,z,\e)\vspace{0.3mm}\\
=&x^{\ov{\kappa}-1}(t^{\ov{\kappa}-1}(z,\e),z,\e)+\int_{
t^{\ov{\kappa}-1}(z,\e)}^t
F_0(s,x(s,z,\e))ds\vspace{0.2cm}\\
&+\e\int_{t^{\ov{\kappa}-1}(z,\e)}^t
F_1(s,x(s,z,\e))ds+\CO(\e^{2}).
\end{array}
\]
Since
\[
\begin{array}{RL}
x^{i}(t^{i}(z,\e),z,\e)=&x^{i-1}(t^{i-1}(z,\e),z,\e)+\int_{t^{i-1}(z,\e)}^
{t^{i}(z,\e)} F_0(t,x(t,z,\e))dt\vspace{0.2cm}\\
&+\e \int_{t^{i-1}(z,\e)}^
{t^{i}(z,\e)} F_1(t,x(t,z,\e))dt+\CO(\e^2),
\end{array}
\]
for $i=1,2,\ldots,\kappa_z$, we obtain, proceeding by induction on $i$, that
\begin{equation}\label{x}
x(t,z,\e)=z+\int_0^tF_0(s,x(s,z,\e))ds+\e\int_0^tF_1(s,x(s,z,\e))ds+\CO(\e^2).
\end{equation}

From here the proof of the lemma follows by proving several claims.

\begin{claim}\label{c1}
There exists a small parameter $\e_0>0$ such that for any $z\in Z$ and for $i=0,1,2,\ldots,\kappa_z$ the function $t^i(z,\e)$ is of class $\C^1$  for every $z$ in a neighborhood $U_z\subset C$ of $z$ and for $\e\in[0,\e_0]$, and $(\p\, t^i/\p\e)(z,0)=0$. Moreover for $t^{i-i}(z,0)\leq t\leq t^i(z,0)$ we have that $y_1(t,z)=(\p\,x^i/\p\e)(t,z,0)$ for $i=1,2,\ldots,\kappa_{z}$.
\end{claim}

First of all we note that $t^1(z,\e)=0$ and $t^{\kappa_z}(z,\e)=T$. So the first part of Claim \ref{c1} is clearly true for $i=0$ and $i=\kappa_z$.

\smallskip

We have concluded above that for each $z\in Z$ the curve $t\mapsto x(t,z,0)$ reaches the discontinuity set only at points of $\Sigma^c$. Let $z^i=x^i(t^i(z,0),z,0)$ and $p^i_z=(t^i(z,0),z^i)\in \Sigma^c$, then $p^i_z\in\Sigma^c$ for
every $i=1,2,\ldots,\kappa_{z}$ if $(0,x(T,z,0))\in\Sigma$, and for every $i=1,2,\ldots,\ov{\kappa}_{z}-1$ if $(0,x(T,z,0))\notin\Sigma$. Particularly $p_i$ is a generic
point of $\Sigma$, so there exists a
neighborhood $G_{p^i_z}$ of $p^i_z$ such that
$\CS_{p^i_z}=G_{p^i_z}\cap\Sigma$ is a $\CC^m$ embedded
hypersurface of $\s^1\times D$ with $m\geq1$. It is well known that
$\CS_{p^i_z}$ can be locally described as the inverse image of a
regular value of a $\CC^m$  function. Thus there exists a small neighborhood $\breve G_{p^i_z}$ of $p^i_z$ with $\breve G_{p^i_z}\subset G_{p^i_z}$ and a $\CC^m$
function $h_i:\breve G_{p^i_z}\rightarrow \R $ such that $\breve G_{p^i_z}\cap\CS_{p^i_z} =h_i^{-1}(0)\cap\Sigma$.

\smallskip

For $(t,x)\in \breve  G_{p^i_z}$ system \eqref{MRs1} can be written
as the autonomous system
\[
\begin{pmatrix}\tau'\\x'\end{pmatrix}=
\begin{cases}
X(\tau,x,\e)\quad \textrm{if}\quad h_i(\tau,x)>0,\vspace{0.2cm}\\
Y(\tau,x,\e)\quad \textrm{if}\quad h_i(\tau,x)<0,
\end{cases}
\]
where
\[
\begin{array}{C}
X(\tau,x,\e)=\left(\begin{array}{C}1 \\ F_0^{n_{i+1}}(\tau,x) \e F_1^{n_{i+1}}(\tau,x)+
\e^2R^{n_{i+1}}(\tau,x,\e)\end{array}
\right),\vspace{0.2mm}\\
Y(\tau,x,\e)=\left(\begin{array}{C}1 \\ F_0^{n_i}(\tau,x)+\e F_1^{n_i}(\tau,x)+\e^2R^{n_i}(\tau,x,\e)\end{array}\right).
\end{array}
\]

From the definition of crossing region we also have $X h_i(p^i_z,0) Y h_i(p^i_z,0)>0$, then
\begin{equation}\label{Yh}
\begin{array}{RL}
0\neq Y h_i(p^i_z,0)&=\left\langle \left(\dfrac{\partial h_i}{\partial t}(p^i_z),
\dfrac{\partial h_i}{\partial x}(p^i_z)\right),\left(1,F_0^{n_{i}}(p^i_z)\right)\right\rangle\vspace{0.2cm}\\
&=\dfrac{\partial h_i}{\partial t}(p^i_z)+\dfrac{\partial h_i}{\partial x}(p^i_z)F_0^{n_{i}}(p^i_z).
\end{array}
\end{equation}

Now defining $H_i(t,\zeta,\e)= h_i(t,x^i(t,\zeta,\e))$ we have that $H_i(t^i(z,0),z,0)=0$. Since
\[
\begin{array}{RL}
\dfrac{\partial H_i}{\partial t}(t^i(z,0),z,0)=&\left.\dfrac{\partial}
{\partial t}h_i(t,x^i(t,\zeta,\e))\right|_{(t,\zeta,\e)=(t^i(z,0),z,0)}\vspace{0.3cm}\\
=& \dfrac{\partial h_i}{\partial t}(t^i(z,0),x^i(t^i(z,0),z,0))\vspace{0.2cm}\\
&+\dfrac{\partial h_i}
{\partial x}(t^i(z,0),x^i(t^i(z,0),z,0))\dfrac{\partial x^i}{\partial t}(t^i(z,0),z,0)\vspace{0.3cm}\\
=& \dfrac{\partial h_i}{\partial t}(p^i_z)+\dfrac{\partial h_i}
{\partial x}(p^i_z)\dfrac{\partial x^i}{\partial t}(t^i(z,0),z,0)\vspace{0.3cm}\\
=&\dfrac{\partial h_i}{\partial t}(p^i_z)+\dfrac{\partial h_i}
{\partial x}(p^i_z)F_0^{n_{i}}(p^i_z)\vspace{0.3cm}\\
=&Y h_i(p^i_z,0)\neq0,
\end{array}
\]
from the Implicit Function Theorem we conclude that there exist a small neighborhood $U_z\subset C$ of $z$ and a small parameter $\tilde{\e}_z>0$ such that $t^i(\zeta,\e)$ is the unique $\CC^m$ function with $H(t^i(\zeta,\e),\e)=0$ for every $\zeta\in U_z$ and $\e\in[0,\tilde{\e}_z]$. So
\begin{equation}\label{ti}
t^i(\zeta,\e)=t^i(\zeta,0)+\e \dfrac{\p\,t^i}{\p\e}(\zeta,0)+o(\e)
\end{equation}
for every
$i=1,2,\ldots,\ov{\kappa}_{z}-1$. Since $Z$ is compact we can take $\e_0$ as a minimum element of $\tilde{\e}_z$'s for $z\in Z$.

\smallskip

Now we shall use finite induction to conclude the proof of Claim \ref{c1}. We note that $h_i(t^i(z,\e), x^i(t^i(z,\e),z,\e))=0$ for $\e\in[0,\e_0]$, so
\begin{equation}\label{ind1}
\begin{array}{RL}
0=&\dfrac{\partial}{\partial \e}h(t^i(z,\e), x^i(t^i(z,\e),z,\e))\Big|_{\e=0}\vspace{0.3cm}\\

=&\dfrac{\partial h}{\partial\, t}(p^i_z)\dfrac{\p\,t^i}{\p\e}(z,0)+\dfrac{\partial h}{\partial z}
(p^i_z)\left(\dfrac{\partial x^i}{\partial t}(t^i(z,0),z,0)\dfrac{\p\, t^i}{\p\e}(z,0)\right.\vspace{0.2cm}\\
&\left.+\dfrac{\partial x^i}{\partial \e}(t^i(z,0),z,0)\right)\vspace{0.3cm}\\

=&\dfrac{\partial h}{\partial\, t}(p^i_z)\dfrac{\p\,t^i}{\p\e}(z,0)+\dfrac{\partial h}{\partial z}
(p^i_z)\left(F_0^{n_i}(p^i_z)\dfrac{\p\, t^i}{\p\e}(z,0)+\dfrac{\partial x^i}{\partial \e}(t^i(z,0),z,0)\right)\vspace{0.3cm}\\

=&\left\langle \nabla h(p^i_z),\Big(\dfrac{\p\,t^i}{\p\e}(z,0), F_0^{n_i}(p^i_z)\dfrac{\p\, t^i}{\p\e}(z,0)+\dfrac{\partial x^i}{\partial \e}(t^i(z,0),z,0)\Big)\right\rangle,
\end{array}
\end{equation}
for $i=1,2,\ldots,\kappa_z$.

\smallskip

Taking $i=1$, from \eqref{xi} we obtain that
\begin{equation}\label{ind2}
\dfrac{d}{dt}\left(\dfrac{\p x^1}{\p \e}(t,z,0)\right)=DF_0^{n_1}(t,x^1(t,z,0))\left(\dfrac{\p x^1}{\p \e}(t,z,0)\right)+F_1^{n_1}(t,x^1(t,z,0)).
\end{equation}
So for $0\leq t\leq t^1(z,0)$ the differential system \eqref{ind2} becomes
\begin{equation}\label{ind3}
\dfrac{d}{dt}\left(\dfrac{\p x^1}{\p\e}(t,z,0)\right)=DF_0(t,x(t,z,0))\left(\dfrac{\p x^1}{\p\e}(t,z,0)\right)+F_1(t,x(t,z,0)).
\end{equation}
Since $\dfrac{\p x^1}{\p \e}(0,z,0)=0$ the solution of the linear differential system \eqref{ind3} is
\begin{equation}\label{lab1}
\dfrac{\p x^1}{\p\e}(t,z,0)=Y(t,z)\int_0^tY(s,z)^{-1} F_1(x(s,z,0))ds=y_1(t,z),
\end{equation}
for $0\leq t\leq t^1(z,0)$. Now  from hypothesis $(H3)$ and from equality \eqref{ind1}, for $i=1$, we have that 
\begin{equation}\label{proc1}
\Big(\la \dfrac{\p\,t^1}{\p\e}(z,0), \la F_0^{n_1}(p^1_z)\dfrac{\p\, t^1}{\p\e}(z,0)+y_1(t^1(z,0),z)\Big)\in T_{p^1_z}\Sigma
\end{equation}
for every $\la\in[0,1]$. Thus
\begin{equation}\label{proc2}
\begin{array}{RL}
0=&\left\langle \nabla h(p^1_z),\Big(\la \dfrac{\p\,t^1}{\p\e}(z,0), \la F_0^{n_1}(p^1_z)\dfrac{\p\, t^1}{\p\e}(z,0)+y_1(t^1(z,0),z)\Big)\right\rangle\vspace{0.2cm}\\
=&\la\left(\dfrac{\partial h}{\partial\, t}(p^1_z)\dfrac{\p\,t^1}{\p\e}(z,0)+\dfrac{\partial h}{\partial z}
(p^1_z)F_0^{n_1}(p^1_z)\dfrac{\p\, t^1}{\p\e}(z,0)\right)+\dfrac{\partial h}{\partial z}
(p^1_z)y_1(t^1(z,0),z)\vspace{0.2cm}\\
=&\la Yh_1(p^1_z,0)\dfrac{\p\,t^1}{\p\e}(z,0)+\dfrac{\partial h}{\partial z}
(p^1_z)y_1(t^1(z,0),z),
\end{array}
\end{equation}
for every $\la\in[0,1]$. Computing the derivative with respect to $\la$ in \eqref{proc2} it follows that $Yh_1(p^1_z,0)\dfrac{\p\,t^1}{\p\e}(z,0)=0$ . So from \eqref{Yh} we obtain that 
\begin{equation}\label{lab2}
\dfrac{\p\,t^1}{\p\e}(z,0)=0.
\end{equation}
Hence from \eqref{lab1} and \eqref{lab2} the claim is proved for $i=1$.

\smallskip

Given a positive integer $\ell>1$, we assume by induction hypothesis that Claim \ref{c1} is true for $i=\ell-1$. Taking $i=\ell$ from \eqref{xi} we have that
\begin{equation}\label{ind4}
\dfrac{d}{dt}\left(\dfrac{\p x^{\ell}}{\p \e}(t,z,0)\right)=DF_0^{n_{\ell}}(t,x^{\ell}(t,z,0))\left(\dfrac{\p x^{\ell}}{\p \e}(t,z,0)\right)+F_1^{n_{\ell}}(t,x^{\ell}(t,z,0)).
\end{equation}
So for $t^{\ell-1}(z,0)\leq t\leq t^{\ell}(z,0)$ the differential system \eqref{ind4} becomes
\begin{equation}\label{ind5}
\dfrac{d}{dt}\left(\dfrac{\p x^{\ell}}{\p\e}(t,z,0)\right)=DF_0(t,x(t,z,0))\left(\dfrac{\p x^{\ell}}{\p\e}(t,z,0)\right)+F_1(t,x(t,z,0)).
\end{equation}
From \eqref{rec} we have that $x^{\ell}(t^{\ell-1}(z,\e),z,\e)=x^{\ell-1}(t^{\ell-1}(z,\e),z,\e)$ for every $\e\in[0,\e_0]$. Computing its derivative with respect to $\e$ at $\e=0$ we obtain that
\begin{equation}\label{ine6}
\begin{array}{L}
\dfrac{\p x^{\ell}}{\p\,t}(t^{\ell-1}(z,0),z,0)\dfrac{\p\, t^{\ell-1}}{\p \e}(z,0)+ \dfrac{\p x^{\ell}}{\p \e}(t^{\ell-1}(z,0),z,0)=\vspace{0.2cm}\\
\dfrac{\p x^{\ell-1}}{\p t}(t^{\ell-1}(z,0),z,0)\dfrac{\p \, t^{\ell-1}}{\p \e}(z,0)+ \dfrac{\p x^{\ell-1}}{\p \e}(t^{\ell-1}(z,0),z,0).
\end{array}
\end{equation}
So from induction hypothesis it follows that
\begin{equation}\label{ine7}
\dfrac{\p x^{\ell}}{\p \e}(t^{\ell-1}(z,0),z,0)=\dfrac{\p x^{\ell-1}}{\p \e}(t^{\ell-1}(z,0),z,0)=y_1(t^{\ell-1},z).
\end{equation}
We note that \eqref{ine7} is the initial condition for system \eqref{ind5}. Thus for $t^{\ell-1}(z,0)\leq t\leq t^{\ell}(z,0)$ regarding to the linear differential equation \eqref{ind5} we get that
\begin{equation}\label{ine8}
\dfrac{\p x^{\ell}}{\p\e}(t,z,0)=\widetilde Y(t,z)y_1(t^{\ell-1}(z,0),z)+\widetilde Y(t,z)\int_0^t\widetilde Y(s,z)^{-1} F_1(x(s,z,0))ds,
\end{equation}
where $\widetilde Y(t,z)$ is the fundamental matrix of the linear differential system \eqref{lin} such that $\widetilde Y(t^{\ell-1}(z,0),z)$
is the identity matrix. Clearly  $\widetilde Y(t,z)=Y(t,z)Y(t^{\ell-1}$ $(z,0),z)^{-1}$. So substituting \eqref{y1} in \eqref{ine8} we get
\begin{equation}\label{ine9}
\begin{array}{RL}
\dfrac{\p x^{\ell}}{\p\e}(t,z,0)=&Y(t,z)\int_0^{t^{\ell-1}(z,0)}Y(s,z)^{-1} F_1(x(s,z,0))ds\vspace{0.2cm}\\
&+Y(t,z)\int_{t^{\ell-1}(z,0)}^t Y(s,z)^{-1} F_1(x(s,z,0))ds\vspace{0.3cm}\\
=&Y(t,z)\int_{t^{\ell-1}(z,0)}^t Y(s,z)^{-1} F_1(x(s,z,0))ds=y_1(t,z),
\end{array}
\end{equation}
for $t^{\ell-1}(z,0)\leq t\leq t^{\ell}(z,0)$. Now repeating the procedure of \eqref{proc1} and \eqref{proc2} for $i=\ell$ we conclude that $\dfrac{\p\,t^{\ell}}{\p\e}(z,0)=0$. So we have proved Claim \ref{c1}.

\bigskip

\begin{claim}\label{c2}
The equality $x(t,z,\e)=x(t,z,0)+\CO(\e)$ holds for every $z\in Z$ and $\e\in[0,\e_0]$.
\end{claim}

If $t\in[t^{\ov{\kappa}-1}(z,\e),t^{\ov{\kappa}}(z,\e))$ then we compute
\[
\begin{array}{RL}
\int_0^t F_0(s,x(s,z,\e))ds =& \sum_{i=1}^{\ov{\kappa}-1}\left(
\int_{t^{i-1}(z,\e)}^{t^{i}(z,\e)}F_0^{n_i}(s,x(s,z,\e))ds\right)\vspace{0.2cm}\\
&+\int_{t^{\ov{\kappa}-1}(z,\e)}^{t}F_0^{n_{\ov{\kappa}}}(s,x(s,z,\e))ds\vspace{0.3cm}\\

=& \sum_{i=1}^{\ov{\kappa}-1}\left(
\int_{t^{i-1}(z,0)}^{t^{i}(z,0)}F_0^{n_i}(s,x(s,z,\e))ds\right)\vspace{0.2cm}\\
&+\int_{t^{\ov{\kappa}-1}(z,0)}^{t}F_0^{n_{\ov{\kappa}}}(s,x(s,z,\e))ds+E_0(\e),\\
\end{array}
\]
where
\[
\begin{array}{RL}
E_0(\e)=&\sum_{i=1}^{\ov{\kappa}-1}\left(\int_{t^{i-1}(z,\e)}^{t^{i-1}(z,0)}
F_0^{n_i}(s,x(s,z,\e))ds-\int_{t^{i}(z,\e)}^{t^{i}(z,0)}F_0^{n_i}
(s,x(s,z,\e))ds \right)\vspace{0.2cm}\\
&+ \int_{t^{\ov{\kappa}-1}(z,\e)}^{t^{\ov{\kappa}-1}(z,0)}F_0^{n_{\ov{\kappa}}}
(s,x(s,z,\e))ds.
\end{array}
\]
It is easy to see that there exists a constant $\ov E$ such that
\begin{equation}\label{E}
||E_0(\e)||\leq \ov E \sum_{i=0}^{\ov{\kappa}-1}|t^i(z,0) - t^i(z,\e)|.
\end{equation}
Indeed the function $F_0^{n_i}(t,x)$ are bounded in the set $\s^1\times K$, so
\[
\begin{array}{RL}
\left|\left|\int_{t^{i}(z,\e)}^{t^{i}(z,0)}F_0^{n_i}
(s,x(s,z,\e))ds\right|\right|&\leq\int_{t^{i}(z,\e)}^{t^{i}(z,0)}\left|\left|F_0^{n_i}
(s,x(s,z,\e))\right|\right|ds\vspace{0.2cm}\\
&\leq L|t^i(z,0) - t^i(z,\e)|,
\end{array}
\]
for $i=0,1,2,\ldots,\ov{\kappa}$.

\smallskip

From Claim \ref{c1} we conclude that $E_0(\e)=o(\e)$, particularly $E_0(\e)=\CO(\e)$. Thus
\begin{equation}\label{FF0}
\begin{array}{RL}
\int_0^t F_0(s,x(s,z,\e))ds=&\sum_{i=1}^{\ov{\kappa}-1}\left(
\int_{t^{i-1}(z,0)}^{t^{i}(z,0)}F_0^{n_i}(s,x(s,z,\e))ds\right)\vspace{0.2cm}\\
&+\int_{t^{\ov{\kappa}-1}(z,0)}^{t}F_0^{n_{\ov{\kappa}}}(s,x(s,z,\e))ds+\CO(\e).
\end{array}
\end{equation}

Using the fact that the functions $F_0^{n_i}$ for $i=1,2,\ldots,\kappa_z$ are locally Lipschitz in the second variable together with \eqref{FF0} we obtain
\[
\begin{array}{RL}
&\left|\left|\int_0^t F_0(s,x(s,z,\e))-F_0(s,x(s,z,0))ds\right|\right|\vspace{0.2cm}\\

\leq&\sum_{i=1}^{\ov{\kappa}-1}
\int_{t^{i-1}(z,0)}^{t^{i}(z,0)}\big|\big|F_0^{n_i}(s,x(s,z,\e))-F_0^{n_i}(s,x(s,z,0))\big|\big|ds\vspace{0.2cm}\\
&+\int_{t^{\ov{\kappa}-1}(z,0)}^{t}\big|\big|F_0^{n_{\ov{\kappa}}}(s,x
(s,z,\e))-F_0^{n_{\ov{\kappa}}}(s,x(s,z,0))\big|\big|ds+\CO(\e)\vspace{0.2cm}\\

\leq& L\sum_{i=1}^{\ov{\kappa}-1}
\int_{t^{i-1}(z,0)}^{t^{i}(z,0)}||x(s,z,\e)-x(s,z,0)||ds\vspace{0.2cm}\\
&+L\int_{t^{\ov{\kappa}-1}(z,0)}^{t}||x(s,z,\e)-x(s,z,0)||ds+\CO(\e)\vspace{0.2cm}\\
=&L\int_{0}^{t}||x(s,z,\e)-x(s,z,0)||ds + \CO(\e).
\end{array}
\]
From \eqref{x} this implies that
\begin{equation}\label{gron}
\begin{array}{RL}
||x(t,z,\e)-x(t,z,0)||\leq&\int_0^t||F_0(s,x(s,z,\e))-F_0(s,x(s,z,0))||ds\\
&+|\e|\int_0^t||F_1(s,x(s,z,\e))||ds+\CO(\e^2)\vspace{0.2cm}\\
\leq&|\e|MT+L\int_{0}^{t}||x(s,z,\e)-x(s,z,0)||ds\vspace{0.2cm}\\
\leq&|\e|M T e^{TL}.\\
\end{array}
\end{equation}
The last inequality is a consequence of Gronwall Lemma (see, for example, Lemma 1.3.1 of \cite{SVM}).

\smallskip

From \eqref{gron} we conclude that $x(t,z,\e)=x(t,z,0)+\CO(\e)$. So we have proved Claim \ref{c2}.

\bigskip

\begin{claim}\label{c3}
The equality $x(t,z,\e)=x(t,z,0)+\e y_1(t,z)+o(\e)$ holds for every $z\in Z$ and $\e\in[0,\e_0]$. 
\end{claim}

In the proof of Lemma 1 of \cite{LNT2} it has been proved that
\begin{equation}\label{F0F1}
\begin{array}{RL}
F_0^{n_i}(t,x^i(t,z,\e))=&F_0^{n_i}(t,x^i(t,z,0))+D_x F^{n_i}_0(t,x^i(t,z,0))\vspace{0.2cm}\\
&\cdot(x^i(t,z,\e)-x^i(t,z,0))+\CO(\e^2), \vspace{0.3cm}\\
F_1^{n_i}(t,x^i(t,z,\e))=&F_1^{n_i}(t,x^i(t,z,0))+\CO(\e),
\end{array}
\end{equation}
for all $t^{i-1}(z,\e)\leq t\leq t^i(z,\e)$ and for every $i=1,2,\ldots,\kappa_z$. In what follows we give a sketch of the proof.

\smallskip

Let $\CL(\mu)=G\big(t,\mu x^i(t,z,\e)+(1-\mu)x^i(t,z,0)\big)$. So
\begin{equation}\label{lab5}
\begin{array}{RL}
F_0^{n_i}(t,x^i(t,z,\e))=&F_0^{n_i}(t,x^i(t,z,0))+\CL_1(1)-\CL_1(0)\vspace{0.3cm}\\

=&F_0^{n_i}(t,x^i(t,z,0))+\int_0^1\CL_1'(\la_{1})d\la_{1}\vspace{0.3cm}\\

=&F_0^{n_i}(t,x^i(t,z,0))+\int_0^1D_x F_0^{n_i}(t,\ell_1(x^i(t,z,\e)))d\la_1\vspace{0.2cm}\\
&\cdot(x^i(t,z,\e)-x^i(t,z,0))\vspace{0.3cm}\\

=&\int_0^1\Big[D_x F_0^{n_i}(t,\ell_1(x^i(t,z,\e)))-D_x F_0^{n_i}(t,x^i(t,z,0))\Big]d\la_1\vspace{0.2cm}\\
&\cdot(x^i(t,z,\e)-x^i(t,z,0))+F_0^{n_i}(t,x^i(t,z,0))\vspace{0.2cm}\\
&+D_x F_0^{n_i}(t,x^i(t,z,0))\cdot(x^i(t,z,\e)-x^i(t,z,0)).
\end{array}
\end{equation}
So observing that the function $D_x F_0^{n_i}(t,x)$ is locally Lipschitz in the variable $x$ and using Claim \ref{c2} in \eqref{lab5} we obtain the equality for $F_0^{n_i}$ of \eqref{F0F1}. The equality for $F_1^{n_i}(t,x)$ of \eqref{F0F1} is obtained directly by using Claim \ref{c2} together with the fact that this function is Lipschitz in the variable $x$.

\smallskip
 
From \eqref{F0F1} we obtain that 
\begin{equation}\label{F0}
\begin{array}{RL}
F_0^{n_i}(t,x^i(t,z,\e))=&F_0^{n_i}(t,x^i(t,z,0))+\e D_x F^{n_i}_0(t,x^i(t,z,0))\vspace{0.2cm}\\
&\cdot\dfrac{\p x^i}{\p\e}(t,z,0)+\CO(\e^2),
\end{array}
\end{equation}
for all $t^{i-1}(z,\e)\leq t\leq t^i(z,\e)$ and for every $i=1,2,\ldots,\kappa_z$. For the moment we cannot use Claim \ref{c1} to ensure that $\dfrac{\p x^i}{\p\e}(t,z,0)=y_1(t,z)$ because it is only true when $t^{i-1}(z,0)\leq t\leq t^i(z,0)$.

\smallskip

Given $z\in C$ we have that, for every $t^{i-1}(z,\e)\leq t\leq t^i(z,\e)$, $x^i(t,z,\e)=x(t,z,\e)$ for $i=1,2,\ldots,\kappa_{\al}$.  Moreover if $t^{i-1}(z,\e)\leq s<t^i(z,\e)$ and $\e\in[0,\e_0]$, then $F_j^{n_i}(s,x^i(s,z,\e))=F_j(s,x(t,z,\e))$ for $j=0,1$ and for every $i=1,2,\ldots,\ov{\kappa}$.

\smallskip

If $t^{\kappa-1}(z,\e)\leq t\leq t^{\kappa}(z,\e)$ from \eqref{F0F1} we compute
\begin{equation}\label{F1}
\begin{array}{L}
\int_0^t F_1(s,x(s,z,\e))ds=\vspace{0.3cm}\\

\left(\sum_{i=1}^{\ov{\kappa}-1}
\int_{t^{i-1}(z,\e)}^{t^{i}(z,\e)}F_1^{n_i}(s,x^i(s,z,\e))ds\right)+
\int_{t^{\ov{\kappa}-1}(z,\e)}^{t}F_1^{n_{\ov{\kappa}}}(s,x^{\ov{\kappa}}
(s,z,\e))ds=\vspace{0.3cm}\\

\left(\sum_{i=1}^{\ov{\kappa}-1}
\int_{t^{i-1}(z,\e)}^{t^{i}(z,\e)}F_1^{n_i}(s,x^i(s,z,0))ds\right)+
\int_{t^{\ov{\kappa}-1}(z,\e)}^{t}F_1^{n_{\ov{\kappa}}}(s,x^{\ov{\kappa}}(s,z,0))ds+\CO(\e)=\vspace{0.3cm}\\

\left(\sum_{i=1}^{\ov{\kappa}-1}\int_{t^{i-1}(z,0)}^{t^{i}(z,0)}
F_1^{n_i}(s,x^i(s,z,0))ds\right)+ \int_{t^{\ov{\kappa}-1}(z,0)}^{t}
F_1^{n_{\ov{\kappa}}}(s,x^{\ov{\kappa}}(s,z,0))ds+E_1(\e)\vspace{0.2cm}\\
+\CO(\e)=\vspace{0.3cm}\\

\left(\sum_{i=1}^{\ov{\kappa}-1}\int_{t^{i-1}(z,0)}^{t^{i}(z,0)}
F_1(s,x(s,z,0))ds\right)+ \int_{t^{\ov{\kappa}-1}(z,0)}^{t}
F_1(s,x(s,z,0))ds+E_1(\e)+\CO(\e)=\vspace{0.3cm}\\

\int_{0}^{t}F_1(s,x(s,z,0))ds+E_1(\e)+\CO(\e),
\end{array}
\end{equation}
where
\[
\begin{array}{RL}
E_1(\e)=&\sum_{i=1}^{\ov{\kappa}-1}\left(\int_{t^{i-1}(z,\e)}^{t^{i-1}(z,0)}
F_1^{n_i}(s,x^i(s,z,0))ds-\int_{t^{i}(z,\e)}^{t^{i}(z,0)}F_1^{n_i}
(s,x^i(s,z,0))ds \right)\vspace{0.2cm}\\
&+ \int_{t^{\ov{\kappa}-1}(z,\e)}^{t^{\ov{\kappa}-1}(z,0)}F_1^{n_{\ov{\kappa}}}
(s,x^{\ov{\kappa}}(s,z,0))ds.
\end{array}
\]
Now, as in the case $E_0(\e)$ of the proof of Claim \ref{c2}, it is easy to see that there exists a constant $\widetilde E$ such that
\begin{equation}\label{E}
||E_1(\e)||\leq \widetilde E \sum_{i=0}^{\ov{\kappa}-1}|t^i(z,0) - t^i(z,\e)|.
\end{equation}
So from Claim \ref{c1} we conclude that $E_1(\e)=o(\e)$ and consenquently $E_1(\e)=\CO(\e)$. Going back to inequality
\eqref{F1} we obtain
\begin{equation}\label{FF1}
\int_0^t F_1(s,x(s,z,\e))ds=\int_0^t F_1(s,x(s,z,0))ds+\CO(\e).
\end{equation}

From Claim \ref{c1}, $\dfrac{\p x^i}{\p\e}(t,z,0)=y_1(t,z)$ for $t^{i-1}(z,0)\leq t\leq t^i(z,0)$, so from \eqref{F0} we compute
\begin{equation}\label{iine1}
\begin{array}{L}
\int_0^t F_0(s,x(s,z,\e))ds =\vspace{0.3cm}\\

\sum_{i=1}^{\ov{\kappa}-1}\left(
\int_{t^{i-1}(z,\e)}^{t^{i}(z,\e)}F_0^{n_i}(s,x^i(s,z,\e))ds\right)+
\int_{t^{\ov{\kappa}-1}(z,\e)}^{t}F_0^{n_{\ov{\kappa}}}(s,x^{\ov{\kappa}}
(s,z,\e))ds=\vspace{0.3cm}\\

\sum_{i=1}^{\ov{\kappa}-1}\left(
\int_{t^{i-1}(z,\e)}^{t^{i}(z,\e)}\big[F_0^{n_i}(s,x^i(s,z,0))+\e D_x
F_0^{n_i}(s,x^i(s,z,0))\dfrac{\p x^i}{\p\e}(s,z,0)\big]ds\right)+\vspace{0.2cm}\\
\int_{t^{\ov{\kappa}-1}(z,\e)}^{t}\big[F_0^{n_{\ov{\kappa}}}(s,x^{\ov{\kappa}}(s,z,0))+\e D_x
F_0^{n_{\ov{\kappa}}}(s,x^{\ov{\kappa}}(s,z,0))\dfrac{\p x^{\ov{\kappa}}}{\p\e}(t,z,0)\big]ds+\CO(\e^2)=\vspace{0.3cm}\\

\sum_{i=1}^{\ov{\kappa}-1}\left(
\int_{t^{i-1}(z,0)}^{t^{i}(z,0)}\big[F_0^{n_i}(s,x^i(s,z,0))+\e D_x
F_0^{n_i}(s,x^i(s,z,0))\dfrac{\p x^i}{\p\e}(t,z,0)\big]ds\right)+\vspace{0.2cm}\\
\int_{t^{\ov{\kappa}-1}(z,0)}^{t}\big[F_0^{n_{\ov{\kappa}}}(s,x^{\ov{\kappa}}(s,z,0))+\e D_x
F_0^{n_{\ov{\kappa}}}(s,x^{\ov{\kappa}}(s,z,0))\dfrac{\p x^{\ov{\kappa}}}{\p\e}(t,z,0)\big]ds+E_2(\e)\vspace{0.2cm}\\
+\CO(\e^2)=\vspace{0.3cm}\\

\sum_{i=1}^{\ov{\kappa}-1}\left(
\int_{t^{i-1}(z,0)}^{t^{i}(z,0)}\big[F_0(s,x(s,z,0))+\e D_x
F_0^{n_i}(s,x(s,z,0))y_1(s,z)\big]ds\right)+\vspace{0.2cm}\\
\int_{t^{\ov{\kappa}-1}(z,0)}^{t}\big[F_0(s,x(s,z,0))+\e D_x
F_0^{n_{\ov{\kappa}}}(s,x(s,z,0))y_1(s,z)\big]ds+E_2(\e)+\CO(\e^2)
\end{array}
\end{equation}
The last equality comes from observing that $F_0^{n_i}(s,x^i(s,z,0))=F_0(s,x(s,z,0))$ for every $s\in[t^{i-1}(z,0),t^{i}(z,0))$ and $i=1,2,\ldots,\ov{\kappa}$. From definition \eqref{nota} the inequality \eqref{iine1} becomes
\begin{equation}\label{ine1}
\begin{array}{RL}
\int_0^t F_0(s,x(s,z,\e))ds=&\int_{0}^{t}\big[F_0(s,x(s,z,0))+\e D_x
F_0(s,x(s,z,0))\vspace{0.2cm}\\
&\cdot y_1(s,z)\big]ds+E_2(\e)+\CO(\e^2).
\end{array}
\end{equation}
Here
\[
\begin{array}{RL}
E_2(\e)=&\sum_{i=1}^{\ov{\kappa}-1}\left(\int_{t^{i-1}(z,\e)}^{t^{i-1}(z,0)}
\big[F_0^{n_i}(s,x^i(s,z,0))+\e D_x
F_0^{n_i}(s,x^i(s,z,0))\dfrac{\p x^i}{\p\e}(t,z,0)\big]ds\right.\vspace{0.2cm}\\

&\left.-\int_{t^{i}(z,\e)}^{t^{i}(z,0)}\big[F_0^{n_i}(s,x(s,z,0))+\e D_x
F_0^{n_i}(s,x^i(s,z,0))\dfrac{\p x^i}{\p\e}(t,z,0)\big]ds \right)\vspace{0.2cm}\\

&+ \int_{t^{\ov{\kappa}-1}(z,\e)}^{t^{\ov{\kappa}-1}(z,0)}\big[F_0^{n_{\ov{\kappa}}}(s,x^i(s,z,0))+\e D_x
F_0^{n_{\ov{\kappa}}}(s,x^i(s,z,0))\dfrac{\p x^i}{\p\e}(t,z,0)\big]ds.
\end{array}
\]
It is easy to see that there exists a constant $\widehat E$ such that
\begin{equation}\label{E}
||E_2(\e)||\leq \widehat E \sum_{i=0}^{\ov{\kappa}-1}|t^i(z,0) - t^i(z,\e)|.
\end{equation}
So from Claim \ref{c1} it follows that $E_2(\e)=o(\e)$. Going back to inequality \eqref{ine1} we have
\begin{equation}\label{ine4}
\begin{array}{RL}
\int_0^t F_0(s,x(s,z,\e))ds=&\int_{0}^{t}F_0(s,x(s,z,0))ds\vspace{0.2cm}\\
&+\e \int_{0}^{t} D_x F_0(s,x(s,z,0))y_1(s,z)ds+o(\e).
\end{array}
\end{equation}

So from \eqref{x}, \eqref{FF1}, and \eqref{ine4} we conclude that
\begin{equation}\label{sb1}
\begin{array}{RL}
x(t,z,\e)=&z+\int_0^t F_0(s,x(s,z,0))ds\vspace{0.2cm}\\
&+\e\int_0^t\big[ D_x
F_0(s,x(s,z,0))y_1(s,z)+F_1(s,x(s,z,0))\big]ds+o(\e)\vspace{0.3cm}\\

=&x(s,z,0)+\e y_1(t,z)+o(\e).
\end{array}
\end{equation}
The last equality 
is a simple consequence of the computations made in Claim \ref{c1}. Indeed from \eqref{ind5} and Claim \ref{c1} if $t^{\ell-1}(z,0)\leq t\leq t^{\ell}(z,0)$, then
\[
y_1(t,z)=y_1(t^{\ell-1}(z,0),0)+\int_{t^{\ell-1}(z,0)}^t\big[D_xF_0(s,x(s,z,0))y_1(s,z)+F_1(s,x(s,z,0))\big]ds.
\]
From here, proceeding by induction on $\ell$, we obtain that
\[
y_1(t,z)=\int_0^t\big[ D_xF_0(s,x(s,z,0))y_1(s,z)+F_1(s,x(s,z,0))\big]ds
\]

This completes the proof of Claim \ref{c3} and, consequently, the proof of the lemma.
\end{proof}

\begin{lemma}\label{l2}
Under the hypothesis of Theorem \ref{MRt1} there exists a compact subset $Z$ of $C$ with $\CZ\subset Z^{\circ}$ such that the solution $x(t,z,0)$ of the unperturbed differential system \eqref{ups} is $\C^1$ in the variable $z$ for every $z\in Z$.  Moreover $(\p x/\p z)(t,z,0)=Y(t,z)Y(0,z)^{-1}$. The set $Z$ is defined in the statement of Lemma \ref{l1} and $Y$ is the fundamental matrix solution of \eqref{lin}.
\end{lemma}
\begin{proof}
It is easy to see that there exists a compact subset $Z$ of $C$ such that $\CZ\subset Z^{\circ}$. Given $z\in Z$ the solution of the unperturbed system \eqref{ups} starting at $z$ is given by \eqref{xx} by taking $\e=0$. Since $C\cap\p\Sigma_0=\emptyset$ and $Z\subset C$, there exists a neighborhood $U^0\subset C$ of $z$ such that for every $\zeta\in U^0$ the local flow of the unperturbed system \eqref{ups} starting at the point $\zeta$ is given by $x^1(t,\zeta,0)$. We know that $(t^i(z),x(t^i(z),z,0))\in\Sigma^c$ for $i=1,2,\ldots,\kappa_z-1$. Since $\Sigma^c$ is an open subset of $\Sigma$ we conclude that there exist $U^i\subset \Sigma^c$ neighborhoods of $x(t^i(z),z,0)$ in $\Sigma$ for $i=1,2,\ldots,\kappa_z-1$. For $i=\kappa_z$ we have that $x(t^{\kappa_z}(z),z,0)=z$, so we take $U^{\kappa_z}=U^0$. Moreover, for each $\zeta\in U^i$ the locally flow of the unperturbed system \eqref{ups} starting in $\zeta$ is given by $x^{i+1}(t,\zeta,0)$ for $i=1,2,\ldots,\kappa_z$. Therefore we can choose a small neighborhood $U_z\subset U^0$ such that for every $\zeta\in U_z$, $x(t^{i}(\zeta),\zeta,0)\in U^i$ for $i=1,2,\ldots,\kappa_z$. Hence we conclude that for each $z\in C$ there exists a small neighborhood $U_z\subset C$ of $z$ such that the solution $t\mapsto x(t,\zeta,0)$ can be written as \eqref{xx} for every $\zeta\in U_z$ having the same number $\kappa_z$ of $\C^1$ pieces.

\smallskip  

Let $\f_n(t,t_0,x_0)$ be the solution of the differential equation $x'=F_0^n(t,x)$ such that $\f_n(t_0,t_0,x_0)=x_0$. From the results of the differential dependence of the solutions we conclude that each of these functions are of class $\C^1$ in the variables $(t,t_0,x_0)$. Indeed the function $F_0^n$ is $\C^1$ for $i=1,2,\ldots,\kappa_z$. From Claim \ref{c1} of the proof of Lemma \ref{l1} for $i=1,2,\ldots,\kappa_z$ the function $t^i(\zeta,\e)$ is of class $\C^1$, for every $\zeta\in U_z$ and $\e\in[0,\e_0]$.

\smallskip

From \eqref{rec} we have that
\begin{equation}\label{in1}
\begin{array}{L}
x^1(t,\zeta,0)=\f_{n_1}(t,0,\zeta) \quad \textrm{and}\vspace{0.2cm}\\
x^i(t,\zeta,0)=\f_{n_i}(t,t^{i-1}(\zeta,0),x^{i-1}(t^{i-1}(\zeta,0),\zeta,0)),
\end{array}
\end{equation}
for $\zeta\in U_z$ and for $i=2,3,\ldots,\kappa_z$. So for $i=1$ the function $(t,\zeta)\mapsto x^1(t,\zeta,0)=\f_{n_1}(t,0,\zeta)$ is $\C^1$. Moreover
for $0\leq t\leq t^1(\zeta,0)$ we have that $\dfrac{\p x^1}{\p z}(t,\zeta,0)=Y(t,\zeta)$. Indeed from \eqref{xi} we have that
\begin{equation}\label{lab4}
\begin{array}{RL}
\dfrac{\p}{\p t}\left(\dfrac{\p x^1}{\p z}(t,\zeta,0)\right)&=D_x F_0^{n_1}(t,x^1(t,\zeta,0))\dfrac{\p x^1}{\p z}(t,\zeta,0)\vspace{0.2cm}\\
&=D_x F_0(t,x(t,\zeta,0))\dfrac{\p x^1}{\p z}(t,\zeta,0),
\end{array}
\end{equation}
for $0\leq t\leq t^1(\zeta,0)$. So solving the linear differential equation \eqref{lab4} we have that the $\dfrac{\p x^1}{\p z}(t,\zeta,0)$ is a fundamental matrix solution of system \eqref{lin} for $0\leq t\leq t^1(\zeta,0)$ and $\zeta\in U_z$. Since $\dfrac{\p x^1}{\p z}(0,\zeta,0)$ is the identity matrix, we conclude that $\dfrac{\p x^1}{\p z}(t,\zeta,0)=Y(t,\zeta) Y(0,z)^{-1}$ for $0\leq t\leq t^1(\zeta,0)$ and $\zeta\in U_z$. 

\smallskip

We assume by induction hypothesis that the function $\zeta\mapsto x^{\ell-1}(t,\zeta,0)$ is $\C^1$  for each $t\in\s^1$, and that for $t^{\ell-2}(\zeta,0)\leq t\leq t^{\ell-1}(\zeta,0)$ the equality $\dfrac{\p x^{\ell}}{\p z}(t,\zeta,0)=Y(t,\zeta) Y(0,z)^{-1}$ holds.

\smallskip

From \eqref{in1} we have that, for $i=\ell$, $x^{\ell}(t,\zeta,0)=\f_{n_{\ell}}(t,t^{\ell-1}(\zeta,0),x^{\ell-1}(t^{\ell-1}(\zeta,0),\zeta,$ $0))$. So the the function $\zeta\mapsto x^{\ell}(t,\zeta,0)$ is $\C^1$  because from the induction hypothesis it is composition of $\C^1$  functions. Now, we have
\[
\begin{array}{RL}
\dfrac{\p}{\p t}\left(\dfrac{\p x^{\ell}}{\p z}(t,\zeta,0)\right)&=D_x F_0^{n_{\ell}}(t,x^{\ell}(t,\zeta,0))\dfrac{\p x^{\ell}}{\p z}(t,\zeta,0)\vspace{0.2cm}\\
&=D_x F_0(t,x(t,\zeta,0))\dfrac{\p x^{\ell}}{\p z}(t,\zeta,0),
\end{array}
\]
for $t^{\ell-1}(\zeta,0)\leq t\leq t^{\ell}(\zeta,0)$. Solving this linear differential equation we get that 
\[
\begin{array}{RL}
\dfrac{\p x^{\ell}}{\p z}(t,\zeta,0)=&Y(t,\zeta)Y(t^{\ell-1}(\zeta,0),\zeta)^{-1}\dfrac{\p x^{\ell}}{\p z}(t^{\ell-1}(\zeta,0),\zeta,0)\vspace{0.2cm}\\
=&Y(t,\zeta)Y(0,\zeta)^{-1},
\end{array}
\]
for  $t^{\ell-1}(\zeta,0)\leq t\leq t^{\ell}(\zeta,0)$ and $\zeta\in U_z$. The last equality comes from the induction hypothesis because 
\[
\dfrac{\p x^{\ell}}{\p z}(t^{\ell-1}(\zeta,0),\zeta,0)=\dfrac{\p x^{\ell-1}}{\p z}(t^{\ell-1}(\zeta,0),\zeta,0)=Y(t^{\ell-1}(\zeta,0),\zeta)Y(0,\zeta)^{-1}.
\] 

\smallskip

The above induction proved that for every $z\in Z$, $x^i(t,z,0)$ is a $\C^1$ function in the second variable and $\dfrac{\p x^i}{\p z}(t,z,0)=Y(t,z)Y(0,z)^{-1}$, provided that $t^{i-1}\leq t\leq t^{i}$. We conclude the proof of the lemma by observing that for $z\in Z$ and $t\in\s^1$ there exists $\ell\in\{1,2,\ldots,\kappa_z\}$ such that $t^{\ell-1}(z,0)\leq t\leq t^{i}(z,0)$, hence $x(t,z,0)=x^{\ell}(t,z,0)$.
\end{proof}

\begin{lemma}\label{l3}
Under the hypotheses of Theorem \ref{MRt1} there exists a small parameter $\ov{\e}\in[0,\e_0]$ such that for every $\e\in[0,\ov{\e}]$ the function $z\mapsto x(T,z,\e)$ is locally Lipshchitz for $z\in Z$. The parameter $\e_0$ is defined in the statement of Lemma \ref{l1} and the set $Z$ is defined in the statement of Lemma \ref{l2}.
\end{lemma}
\begin{proof}
From Lemma \ref{l2} we have that for each $z\in Z$ there exists a small neighborhood $U_z\subset C$ of $z$ such that the solution $t\mapsto x(t,\zeta,0)$ can be written as \eqref{xx} for every $\zeta\in U_z$ having the same number $\kappa_z$ of $\C^1$ pieces. Therefore applying the result of the continuous dependence of the solutions on the parameters in each differentiable piece we conclude that for each $z\in Z$ there exists a small neighborhood $\U_z\subset U_z$ and a small parameter $\ov{\e}_z\in(0,\e_0]$ such that the solution $t\mapsto x(t,\zeta,\e)$ can be written as \eqref{xx} for every $\zeta\in \U_z$ and for each $\e\in(0,\e_z]$ having the same number $\kappa_z$ of $\C^1$ pieces. Since $Z$ is a compact set we can choose $\ov{\e}$ a minimal parameter of $\ov{\e}_z$ for $z\in Z$ such that the above result holds for every $\e\in[0,\ov{\e}]$.

\smallskip

Let $\psi_n(t,t_0,x_0,\e)$ be the solution of the differential equation 
\begin{equation}\label{part}
x'=F^n(t,x)=F_0^n(t,x)+\e F_1^n(t,x)+\e^2 R^n(t,x,\e),
\end{equation}
such that $\psi_n(t_0,t_0,x_0,\e)=x_0$. Clearly $\psi_n(t,t_0, $ $x_0,0)=\f_n(t,t_0,x_0)$ which has been defined in Lemma \ref{l2}. From the result of the continuous dependence of the solutions on the initial conditions we conclude that each of these functions are continuous in the variables $(t,t_0,x_0)$. Indeed $F^n$ is a continuous function which is Lipschitz in the second variable for $i=1,2,\ldots,\kappa_z$. Moreover using the Gronwall Lemma (see, for instance, \cite{SVM}) we conclude that
\begin{equation}\label{lips}
||\psi_n(t,s_1,s_1,\e)-\psi_n(t,t_2,z_2,\e)||\leq Me^{LT}|t_1-t_2|+e^{LT}||x_1-x_2||,
\end{equation}
for each $t,s_1,s_2\in\s^1$, $z_1,z_2\in \U_z$, and $\e\in[0,\ov{\e}]$, 
where the constant $L$ and $M$ are defined in the proof of Lemma \ref{l1}. 
From the flow properties of the solutions of system \eqref{part} we have that the equality
\begin{equation}\label{flow}
\psi_n(t+s,t_0,x,\e)=\psi_n(t,t_0,\psi_n(s+t_0,t_0,x,\e),\e)
\end{equation}
holds for every $n=1,2,\ldots,N$.

\smallskip

Given $t_1,t_2,s_1,s_2\in\s^1$, $z_1,z_2\in U_z$ and $\e\in[0,\ov{\e}]$ we can prove that the inequality
\begin{equation}\label{mine}
\begin{array}{RL}
||\psi_n(t_1,s_1,z_1,\e)-\psi_n(t_2,s_2,z_2,\e)||\leq&Me^{LT}|t_1-t_2|+Me^{LT}|s_1-s_2|\vspace{0.2cm}\\
&+e^{LT}||z_1-z_2||.
\end{array}
\end{equation}
holds for $n=1,2,\ldots,N$. Indeed, from \eqref{lips} and \eqref{flow}
\[
\begin{array}{L}
||\psi_n(t_1,s_1,z_1,\e)-\psi_n(t_2,s_2,z_2,\e)||=\vspace{0.3cm}\\
||\psi_n(t_1,s_1,z_1,\e)-\psi_n(t_1,s_2,\psi_n(t_2-t_1+s_2,s_2,z_2,\e),\e)||\leq\vspace{0.3cm}\\
e^{LT}\Big(||z_1-\psi_n(t_2-t_1+s_2,s_2,z_2,\e)||+M|s_1-s_2|\Big)=\vspace{0.3cm}\\
e^{LT}\Big(||\psi_n(t_2-t_1+s_2,t_2-t_1+s_2,z_1,\e)-\psi_n(t_2-t_1+s_2,s_2,z_2,\e)||\vspace{0.2cm}\\
+M|s_1-s_2|\Big)\leq\vspace{0.3cm}\\
e^{LT}\left(||z_1-z_2||+M|t_1-t_2|+M|s_1-s_2|\right).
\end{array}
\]

\smallskip

Again from \eqref{rec} we obtain
\begin{equation}\label{in1}
\begin{array}{L}
x^1(t,\zeta,\e)=\psi_{n_1}(t,0,\zeta,\e) \quad \textrm{and}\vspace{0.2cm}\\
x^i(t,\zeta,0)=\psi_{n_i}(t,t^{i-1}(\zeta,\e),x^{i-1}(t^{i-1}(\zeta,\e),\zeta,\e),\e),
\end{array}
\end{equation}
for $\zeta\in \U_z$ and for $i=2,3,\ldots,\kappa_z$. Thus from \eqref{in1} for $i=1$ the function $x^1(t,\zeta,\e)=\f_{n_1}(t,0,\zeta)$. So from \eqref{mine} we have that
\[
\begin{array}{RL}
||x^1(t_1,z_1,\e)-x^1(t_2,z_2,\e)||=&||\psi_{n_1}(t_1,0,z_1,\e)-\psi_{n_1}(t_2,0,z_2,\e)||\vspace{0.2cm}\\
\leq&e^{LT}\left(||z_1-z_2||+M|t_1-t_2|\right),
\end{array}
\]
for every $z_1,z_2\in \U_z$, $0\leq t_1\leq t^1(z_1,\e)$, $0\leq t_2\leq t^1(z_2,\e)$, and $\e\in[0,\ov{\e}]$.

\smallskip

We assume by induction hypothesis that there exist constants $A_{\ell-1}$ and $B_{\ell-1}$ such that
\[
||x^{\ell-1}(t_1,z_1,\e)-x^{\ell-1}(t_2,z_2,\e)||\leq A_{\ell-1}|t_1-t_2|+B_{\ell-1}||z_1-z_2||,
\]
for every $z_1,z_2\in \U_z$, $t^{\ell-2}(z_1,\e)\leq t_1\leq t^{\ell-1}(z_1,\e)$, $t^{\ell-2}(z_2,\e)\leq t_2\leq t^{\ell-1}(z_2,\e)$, and $\e\in[0,\ov{\e}]$.

\smallskip

From \eqref{in1} we have, for $i=\ell$, that $x^{\ell}(t,\zeta,\e)=\psi_{n_{\ell}}(t,t^{\ell-1}(\zeta,\e),x^{\ell-1}(t^{\ell-1}(\zeta,\e),\zeta,$ $\e),\e)$ for $\zeta\in\U_z$, $t^{\ell-1}(\zeta,\e)\leq t\leq t^{\ell}(\zeta,\e)$ and $\e\in[0,\ov{\e}]$. So from induction hypothesis we obtain that
\begin{equation}\label{inequal1}
\begin{array}{L}
||x^{\ell}(t_1,z_1,\e)-x^{\ell}(t_2,z_2,\e)||=||\psi_{n_{\ell}}(t_1,t^{\ell-1}(z_1,\e),x^{\ell-1}(t^{\ell-1}(z_1,\e),z_1,$ $\e),\e)\vspace{0.2cm}\\
-\psi_{n_{\ell}}(t_2,t^{\ell-1}(z_2,\e),x^{\ell-1}(t^{\ell-1}(z_2,\e),z_2,$ $\e),\e)||\leq\vspace{0.3cm}\\

Me^{LT}|t_1-t_2|+Me^{LT}|t^{\ell-1}(z_1,\e)-t^{\ell-1}(z_2,\e)|\vspace{0.2cm}\\
+e^{LT}||x^{\ell-1}(t^{\ell-1}(z_1,\e),z_1,$ $\e)-x^{\ell-1}(t^{\ell-1}(z_2,\e),z_2,$ $\e)||\leq\vspace{0.3cm}\\

Me^{LT}|t_1-t_2|+e^{LT}(M+A_{\ell-1})|t^{\ell-1}(z_1,\e)-t^{\ell-1}(z_2,\e)|+e^{LT}B_{\ell-1}||z_1-z_2||
\end{array}
\end{equation}
for every $z_1,z_2\in \U_z$, $t^{\ell-1}(z_1,\e)\leq t_1\leq t^{\ell}(z_1,\e)$, $t^{\ell-1}(z_2,\e)\leq t_2\leq t^{\ell}(z_2,\e)$, and $\e\in[0,\ov{\e}]$.

\smallskip

From Claim 1 of the proof of Lemma \ref{l1} we have that $t^{\ell-1}(z,\e)$ is a $\C^1$ function, then there exists a constant $\de>0$ such that $|t^{\ell-1}(z_1,\e)-t^{\ell-1}(z_2,\e)|\leq \de ||z_1-z_2||$ for every $\e\in[0,\ov{\e}]$. Going back to the inequality \eqref{inequal1} we get
\[
||x^{\ell}(t_1,z_1,\e)-x^{\ell}(t_2,z_2,\e)||\leq A_{\ell}|t_1-t_2|+B_{\ell}||z_1-z_2||,
\]
for every $z_1,z_2\in \U_z$, $t^{\ell-1}(z_1,\e)\leq t_1\leq t^{\ell}(z_1,\e)$, $t^{\ell-1}(z_2,\e)\leq t_2\leq t^{\ell}(z_2,\e)$, and $\e\in[0,\ov{\e}]$, where $A_{\ell}=Me^{LT}$ and $B_{\ell}=e^{LT}\left(\de (M+A_{\ell-1})+B_{\ell-1}\right)$.

\smallskip

We conclude the proof of the lemma by observing that $x(T,z,\e)=x^{\kappa_z}(T,z,\e)$ which, from the above induction, is locally Lipschitz in the variable $z$.
\end{proof}

\begin{lemma}\label{l4}
Under the hypothesis of Theorem \ref{MRt2} the solution $x(t,z,\e)$ of the unperturbed differential system \eqref{ups} is $\C^2$ in the variable $z$ for every $z\in Z$.  Moreover $(\p x/\p z)(t,z,0)=Y(t,z)Y(0,z)^{-1}$. The set $Z$ is defined in the statement of Lemma \ref{l2} and $Y$ is the fundamental matrix solution of \eqref{lin}.
\end{lemma}

\begin{proof}
Assuming the hypothesis $(h1)$ instead $(H1)$ we can prove analogously to Claim \ref{c1} in the proof of Lemma \ref{l1} that given $z\in Z$ the function $t^i(z,\e)$ for $i=0,1,2,\cdots,\kappa_z$ is of class $\C^2$ for every $\zeta$ in a neighborhood $U_z\subset C$ of $z$ and $\e\in[0,\e_0]$. Then the proof of the lemma follows analogous the proof of Lemma \ref{l2} but considering the functions $\psi_n(t,t_0,x_0,\e)$ defined in Lemma \ref{l3}.
\end{proof}

The next two lemmas are versions of the so called {\it Lyapunov--Schmidt reduction} for finite dimensional function (see for instance \cite{C}) and its proof can be found in \cite{BLM} and \cite{BGL}, respectively. The first lemma will be used for proving Theorem \ref{MRt1}, and the second one will be used for proving Theorem \ref{MRt2}.

\begin{lemma}\label{LS}
Let $P:\R^d\rightarrow\R^d$ be a $\C^1$ function, and let $Q:\R^d\times[0,\e_0]\rightarrow\R^d$ be a continuous functions which is locally Lipschitz in the first variable, and define $f:\R^d\times[0,\e_0]\rightarrow\R^d$ as $f(z,\e)=P(z)+\e Q(z,\e)$. We assume that there exists an open and bounded subset $V\subset \R^k$ with $k\leq n$ and a $\CC^1$ function $\be_0:\ov V\rightarrow\R^{d-k}$ such that $P$ vanishes on the set $\CZ=\{z_{\al}=(\al,\be_0(\al)):\,\al\in\ov V\}$ and that for any $\al\in \ov V$ the matrix $D P(z_{\al})$ has in its upper right corner the null $k\times(d-k)$ matrix and in the lower corner the $(d-k)\times(d-k)$ matrix $\Delta_{\al}$ with $\det(\Delta_{\al})\neq0$. For any $\al\in\ov V$ we define $f_1(\al)=\pi Q(z_{\al},0)$. Thus if $f_1(\al)\neq0$ for all $\al\in\p V$ and $d(f_1,V,0)\neq 0$, then there exists $\e_1>0$ sufficiently small such that for each $\e\in(0,\e_1]$ there exists at least one $z_{\e}\in\R^d$ with $F(z_{\e},\e)=0$ and $\dis(z_{\e},\CZ)\to 0$ as $\e\to 0$. 
\end{lemma}

\begin{lemma}\label{LS1}
Let $P:\R^d\rightarrow\R^d$ and $Q:\R^d\times[0,\e_0]\rightarrow\R^d$ be $\C^2$ functions, and define $f:\R^d\times[0,\e_0]\rightarrow\R^d$ as $f(z,\e)=P(z)+\e Q(z,\e)$. We assume that there exists an open and bounded subset $V\subset \R^k$ with $k\leq n$ and a $\C^2$ function $\be_0:\ov V\rightarrow\R^{d-k}$ such that $P$ vanishes on the set $\CZ=\{z_{\al}=(\al,\be_0(\al)):\,\al\in\ov V\}$ and that for any $\al\in \ov V$ the matrix $D P(z_{\al})$ has in its upper right corner the null $k\times(d-k)$ matrix and in the lower corner the $(d-k)\times(d-k)$ matrix $\Delta_{\al}$ with $\det(\Delta_{\al})\neq0$. For any $\al\in\ov V$ we define $f_1(\al)=\pi Q(z_{\al},0)$. Thus if there exists $a\in V$ with $f_1(a)\neq0$ and $\det(f'(a))\neq0$, then there exists $\al_{\e}$ such that $f(z_{\al_{\e}},\e)=0$ and $z_{\al_{\e}}\to z_a$ as $\e\to 0$.
\end{lemma}

Now we are ready to prove our main results.

\begin{proof}[Proof of Theorem \ref{MRt1}]
We consider the $\C^1$ function $f:Z\times [0,\e_0]\rightarrow \R^d$, given by
\begin{equation}\label{f}
f(z,\e)=x(T,z,\e)-z.
\end{equation}
Its differentiability comes from Lemma \ref{l2}. Clearly system \eqref{MRs1} for $\e=\ov{\e}\in[0,\e_0]$ has a periodic solution passing through
$\ov z\in C$ if and only if $f(\ov z,\ov{\e})=0$.
 
\smallskip


From Lemma \ref{l1} we have that
\begin{equation}\label{xtay}
x(t,z,\e)=x(t,z,0)+\e y_1(t,z)+o(\e).
\end{equation}
Taking $P(z)=x(t,z,0)-z$ and $Q(z,\e)=y_1(t,z)+\tilde{o}(\e)$, thus $f(z,\e)=P(z)+\e Q(z,\e)$. Moreover from Lemma \ref{l2} $P(z)$ is a $\C^1$ function, and from Lemma \ref{l3} $Q(z,\e)$ is a continuous function which is locally Lipschitz in the first variable because $Q(z,\e)=(x(T,z,\e)-x(T,z,0))/\e$.

In order to apply Lemma \ref{LS} to function \eqref{f} we compute
\[
P(z_{\al})=x(T,z_{\al},0)-z_{\al}=0,
\]
and
\[
\begin{array}{RL}
\dfrac{\p P}{\p z}(z_{\al})=&\dfrac{\p x}
{\p z}(T,z_{\al},0)-Id\vspace{0.2cm}\\
=&Y_{\al}(T)Y_{\al}(0)^{-1}-Id.
\end{array}
\]
So from hypothesis $(H)$ the function $P$ vanishes on the set $\CZ$ and from hypothesis $(H2)$ for any $\al\in\ov V$ the matrix $DP(z_{\al})$ has in its upper right corner the null $k\times(d-k)$ matrix and in the lower corner the $(d-k)\times(d-k)$ matrix $\Delta_{\al}$ with $\det(\Delta_{\al})\neq 0$. Since $\pi Q(\al,\be_0(\al))=\pi y_1(T,z_{\al})=f_1(\al)$, so the proof follows applying Lemma \ref{LS}.
\end{proof}

\begin{proof}[Proof of Theorem \ref{MRt2}]
The proof is analogous to the proof of Theorem \ref{MRt1} applying Lemma \ref{l4} instead of Lemmas \ref{l2} and \ref{l3}, and applying Lemma \ref{LS1} instead of Lemma \ref{LS}. 
\end{proof}
\section{Proof of Proposition \ref{p1}}\label{PP}

\begin{proof}[Proof of Proposition \ref{p1}]
Proceeding with the change of variables $(u,v,w)=(r\cos\T,$ $r\sin\T,z)$ and taking $\T$ as the new time by doing $r'=\dot{r}/\dot{\T}$ and $z'\dot{z}/\dot{\T}$ we obtain
\begin{equation}\label{qs1}
(r',z')=\left\{
\begin{array}{L}
(0,z)+\e\,G^+(\T,r,z)+\CO(\e^2) \quad \textrm{if}\quad 0\leq \T\leq \pi, \\
(0,z)+\e\,G^-(\T,r,z)+\CO(\e^2) \quad \textrm{if}\quad \pi\leq\T\leq 2\pi, \\
\end{array}\right.
\end{equation}
where $G^{\pm}=\left(G^{\pm}_1,G^{\pm}_2\right)$, and
\[
\begin{array}{RL}
G_1^{\pm}=&b_1^{\pm} r\cos^2\T+\left(a_1^{\pm}+d_1^{\pm}z+(b_2^{\pm}+c_1^{\pm})r\sin\T\right)\cos\T\vspace{0.2cm}\\
&+\left(a_2^{\pm}+d_2^{\pm}z+c_2^{\pm} r\sin\T\right)\sin\T,\vspace{0.3cm}\\

G_2^{\pm}=&\dfrac{1}{r}\left(r(a_3^{\pm}+d_3^{\pm}z)-b_2^{\pm}rz\cos^2\T+(C_3^{\pm}r^2+(a_1^{\pm}+d_1^{\pm}z)z+c_1^{\pm}r\sin\T)\sin\T\vspace{0.2cm}\right.\\
&\left.(b_3^{\pm}r^2-(a_2^{\pm}+d_2^{\pm})z+(b_1^{\pm}-c_2^{\pm})rz\sin\T)\cos\T\right).
\end{array}
\]
Here the prime denotes the derivative with respect to $\T$.

\smallskip

For system \eqref{qs1} we have that $D=\{(r,z):\,r>0,\,z\in\R\}$ and $T=2\pi$. We note that $\Sigma=\{(0,r):\,r>0\}\cup\{(\pi,r):\,r>0\}\cup\{(2\pi,r):\,r>0\}$, thus taking $h(\T,r,z)=\T(\T-\pi)(\T-2\pi)$ it follows that $\Sigma=h^{-1}(0)$.

\smallskip

In what follows we shall study the elements of hypothesis $(H)$ of Theorem \ref{MRt1}. For $\e=0$ the solution $x(\T,r,z,0)$ of system \eqref{qs1} such that $x(0,r,z,0)=(r,z)$ is given by $x(\T,r,z,0)=(r,e^{\T}z)$. Taking $V=\{r\in\R:\,r_1<\al<r_2\}$ with $r_1>0$ arbitrarily small and $r_2>r_1$ arbitrarily large, and $\be_0=0$ we have that the solution $x_{\al}(\T)=(\al,0)$ is constant for every $\al\in\ov V$, particularly $2\pi$--periodic. In this case the manifold $\CZ$ of periodic solution of the system \eqref{qs1} when $\e=0$ is given by $\CZ=\{(\al,0):\,r_1\leq\al\leq r_2\}$, and $\Sigma_0=D$. Since $\CZ\subset \Sigma_0$ it follows that $\CZ\cap\p\Sigma_0=\emptyset$. Moreover computing the crossing region of system \eqref{qs1} for $\e>0$ sufficiently small we conclude that $\Sigma^c=\Sigma$, so we obtain that $\widetilde{\CZ}_0\cap\Sigma\subset\Sigma^c$. Therefore hypothesis $(H)$ hods for system \eqref{qs1}.

\smallskip

Hypothesis $(H1)$ of Theorem \ref{MRt1} clearly holds for system \eqref{qs1}. To verify hypothesis $(H2)$ we take
\[
Y(\T,r,z)=\dfrac{\p x}{\p z}(\T,r,z,0)=\left(\begin{array}{CC}1&0\\0&e^{\T}\end{array}\right)
\]
as the fundamental matrix solution of system \eqref{lin} in the case of system \eqref{qs1}. So
\[
Y_{\al}(2\pi)Y_{\al}(0)^{-1}-Id=Y(2\pi,\al,0)Y(0,\al,0)^{-1}-Id=\left(\begin{array}{CC}0&0\\0&e^{2\pi}-1\end{array}\right).
\]
Since $\Delta_{\al}=e^{2\pi}-1\neq 0$ for every $\al\in\ov V$ it follows that hypothesis $(H2)$ holds for system \eqref{qs1}.

\smallskip

Now if $(\T,r,z)\in\Sigma$, then $\T\in\{0,\pi\}$. On the other hand $\nabla h(0,r,z)=(2\pi^2,0,0)$ and $\nabla h(\pi,r,z)=(-\pi^2,0,0)$ for every $(r,z)\in D$. So $\langle\nabla h(\T,r,z),(0,v)\rangle=0$ for every $\T\in\{0,\pi\}$, $(r,z)\in D$, and $v\in\R^2$, which means that for any $v\in\R^2$ we have that $(0,v)\in T_{(\T,r,z)}\Sigma$ for every $\T\in\{0,\pi\}$ and $(r,z)\in D$. In short hypothesis $(H3)$ holds for system \eqref{qs1}.

\smallskip

Using an algebraic manipulator as Mathematica or Maple we compute
\[
f_1(\al)=\dfrac{\pi}{2}\left(b_1^++b_1^-+c_2^++c_2^-\right)\al+2\left(a_2^+-a_2^-\right).
\]
From hypotheses $\left(b_1^++b_1^-+c_2^++c_2^-\right)\left(a_2^--a_2^+\right)>0$, thus
\[
a=\dfrac{4\left(a_2^--a_2^+\right)}{\pi\left(b_1^++b_1^-+c_2^++c_2^-\right)}
\]
is a solutions of the equation $f_1(\al)=0$ such that $f_1'(a)\neq0$. From Remark \ref{bc} it is a sufficient condition to guarantee the existence of a small neighborhood $W\subset V$ of $a$ such that $d(f_1,W,0)\neq 0$. Since $f_1$ is linear, it is clear that $f_1(\al)\neq 0$ for every $\al\in\p W$. Therefore hypothesis $(H4)$ of Theorem \ref{MRt1} holds for system \eqref{qs1}.

\smallskip

Now the proof of the proposition follows directly by applying Theorems \ref{MRt1} and \ref{MRt2}.
\end{proof}

\section*{Acknowledgements}

The first author is partially supported by a MINECO/FEDER grant
MTM2008--03437, an AGAUR grant number 2014SGR 568, an ICREA
Academia, FP7--PEOPLE--2012--IRSES--316338 and 318999, and
FEDER/UNAB10--4E--378. The second author is partially supported by a
FAPESP grant 2013/16492--0 and by a CAPES CSF-PVE grant 88881.030454/2013-01.

\end{document}